\documentclass[11pt,reqno]{amsart}
\usepackage[latin1]{inputenc}
\usepackage{amsfonts}
\usepackage{bbm}

\usepackage{mathrsfs}

\usepackage{amsfonts,amssymb,amsmath,amsthm}
\usepackage{url}
\usepackage{enumerate}
\usepackage[pdftex,bookmarksnumbered]{hyperref}
\usepackage{graphicx,xcolor}
\newcommand{\ds}{\displaystyle}
\newtheorem{theorem}{Theorem}[section]
\newtheorem{lemma}[theorem]{Lemma}
\newtheorem{proposition}[theorem]{Proposition}
\newtheorem{corollary}[theorem]{Corollary}
\theoremstyle{definition}
\newtheorem{definition}[theorem]{Definition}
\newtheorem{remark}{Remark}
\numberwithin{equation}{section}

\newtheorem{example}{Example}

\newtheorem{assumption}{Assumption}

\usepackage[left=2.0cm, right=2.0cm, top=2.0cm, bottom=2.0cm]{geometry}

\usepackage{graphicx}%
\usepackage{color}
\usepackage{cite}
\usepackage{hyperref}
\hypersetup{
	colorlinks = true,%
	citecolor = blue,%[rgb]{0.65,0.0,0.0},
	filecolor=red,%
	linkcolor = [rgb]{0.65,0.0,0.0},%
	anchorcolor = red,
	pagecolor = red,
	urlcolor= [rgb]{0.65,0.0,0.0}
	linktocpage=true,
	pdfpagelabels=true,
	bookmarksnumbered=true,
}

\DeclareMathOperator{\id}{Id}

\DeclareMathOperator{\vol}{vol}
\DeclareMathOperator{\Ca}{Cap}

\DeclareMathOperator{\di}{div}\DeclareMathOperator{\Exp}{Exp}

\DeclareMathOperator{\lo}{loc}

\DeclareMathOperator{\sgn}{sgn}

\DeclareMathOperator{\supp}{supp}

\DeclareMathOperator{\tr}{{tr}}

\DeclareMathOperator{\dvol}{dvol}

\DeclareMathOperator{\dmu}{d\mu}

\author{Yunxia Chen}
\address{
Department of Mathematics\\
East China University of Science and Technology\\
Shanghai, China}
\email{yxchen76@163.com}

\author{Naichung Conan Leung}
\address{The Institute of Mathematical Sciences and Department of Mathematics\\
 The
Chinese University of Hong Kong\\
Shatin, N.T., Hong Kong}
\email{leung@math.cuhk.edu.hk}

\author{Wei Zhao}
\address{
Department of Mathematics\\
East China University of Science and Technology\\
Shanghai, China}
\email{szhao\underline{ }wei@yahoo.com}

\begin{document}

\title[Sharp Hardy inequalities via Riemannian submanifolds]{Sharp Hardy inequalities via Riemannian submanifolds}

\begin{abstract}
This paper is devoted to Hardy inequalities concerning distance functions from  submanifolds of arbitrary codimensions in the Riemannian setting. On a Riemannian manifold with non-negative curvature, we establish several sharp  weighted Hardy inequalities in the cases when  the submanifold is compact as well as non-compact. In particular, these inequalities remain valid even if the ambient manifold is compact, in which case we find an optimal  space of smooth functions to study Hardy inequalities. Further examples are also provided. Our results complement in several aspects those obtained recently in the Euclidean and Riemannian
settings.
\end{abstract}
\maketitle

\section{Introduction} \label{sect1}
It is well-known that Hardy inequalities  play a prominent role in the theory of linear and nonlinear partial differential equations (cf. \cite{BCC,BV,CM2,DA,V,VZ,FKV}, etc.).
The classical Hardy inequality states that for any natural number $m\geq 2$ and real number $p>1$,
\[
\int_{\Omega}|\nabla u|^p dx\geq \left| \frac{p-m}{p} \right|^p \int_{\Omega}\frac{|u|^p}{|x|^p}dx,\ \forall\,u\in C^\infty_0(\Omega\backslash\{\textbf{0}\}),\tag{1.1}\label{basisHARDYFORECUL}
\]
where $\Omega\subset \mathbb{R}^m$ is a (possibly unbounded) domain containing the origin $\mathbf{0}$. In particular,
  the constant $\left|  {(p-m)}/{p} \right|^p$ is sharp (see for instance   Hardy et al. \cite{HPL}). On the other hand, for a (possibly unbounded) {\it convex} domain $\Omega\subset \mathbb{R}^m$ with smooth  compact boundary, the Hardy inequality asserts
  \[
  \int_{\Omega}|\nabla u|^p dx\geq \left( \frac{p-1}{p} \right)^p\int_\Omega \frac{|u|^p}{r^p}dx,\ \forall\,u\in C^\infty_0(\Omega),\tag{1.2}\label{distancebasisHardyforEcu}
  \]
where $r(x):=d(\partial\Omega,x)$ is the distance between  $\partial\Omega$ and $x$. Particularly, the constant $( ({p-1})/{p} )^p$ is still sharp (see Marcus et al. \cite{MMP}  for example). Moreover, several authors showed that (\ref{distancebasisHardyforEcu}) remains valid  for  more general domains if certain analytic conditions hold. Refer to \cite{AK,F,FMT,FMT2,MS} and the bibliography therein. However, the convexity assumption is restrictive while
 analytic conditions are hard to verify. Thus a mild geometric assumption, {\it weak mean convexity}, was introduced in Lewis, Li and Li \cite{LLL}, in which case (\ref{distancebasisHardyforEcu}) remains optimal.
The convexity requires the non-negativity of all the principal curvatures of the boundary, whereas the weakly mean convexity only needs the non-negativity of their trace, the mean curvature. Furthermore,  Psaradakis in \cite{Ps}  showed that weakly mean convexity is equivalent to $\Delta r\leq 0$.
 These facts lead to a further study of the relation between boundary curvature and the Hardy inequality, cf. Balinsky et al. \cite{BEL} and Filippas et al. \cite{FMT3}, etc.

From a point of view of submanifold geometry,  both (\ref{basisHARDYFORECUL}) and (\ref{distancebasisHardyforEcu}) are the Hardy inequalities concerned with the distance from closed submanifolds of special codimensions. For the general case, a sharp Hardy inequality was established in Barbatis, Filippas and Tertikas \cite{BFT}. More precisely, let $\Omega$ be a domain in $\mathbb{R}^m$ and let $N$ be a surface/closed submanifold of codimension $k$ satisfying one of the following conditions:

\smallskip

(a) $k=m$ and $N=\{\mathbf{0}\}\subset \Omega$;\ \ (b) $k=1$ and  $N=\partial\Omega$;\ \ (c) $2\leq k\leq m-1$ and $\Omega \cap N\neq\emptyset$.

\smallskip

\noindent Given $p>1$ with $p\neq k$, suppose  that $r(x):=d(N,x)$ satisfies
\[
 \Delta_p \,r^{\frac{p-k}{p-1}}\leq 0 \ \ \ \text{ in }\Omega\backslash N,\tag{1.3}\label{pLaplacecondition}
\]
where $\Delta_p$ is the $p$-Laplace operator. Then there holds
\[
\int_\Omega |\nabla u|^pdx\geq \left|\frac{p-k}{p}\right|^p\int_{\Omega}\frac{|u|^p}{r^p}dx,\ \forall u\in C^\infty_0(\Omega\backslash N).\tag{1.4}\label{unversHardyinequlityforEuc}
\]
In particular, the constant $\left| {(p-k)}/{p}\right|^p$ is best providing $\sup_{x\in \Omega}r(x)<+\infty$. Moreover, there are several optimal improved versions of (\ref{unversHardyinequlityforEuc}) in Barbatis et al. \cite{BFT}. Also refer to \cite{CM,DD,FMT}, etc., for further results.

Although the general Hardy inequality has been established,  there are a number of issues which are unclear yet. For example, as (\ref{unversHardyinequlityforEuc}) becomes (\ref{basisHARDYFORECUL}) and (\ref{distancebasisHardyforEcu}) when $N=\{\mathbf{0}\}$ and $N=\partial\Omega$ respectively, in both cases
 the finiteness of $\sup_{x\in \Omega}r(x)$ is unnecessary for the sharpness.  So this raises a nature question as to whether  (\ref{unversHardyinequlityforEuc}) remains optimal  even if $r(x)$ is unbounded.
 Besides, since (\ref{pLaplacecondition}) is not easy to check,  another interesting question is how to replace this assumption  by a geometric characterization.
 However, up to now, limited work has been done to study  these problems.

On the other hand,    there has been recently tremendous interest in developing the Hardy inequalities in
the Riemannian framework.
 As far as we know, Carron \cite{Ca} was the first who studied weighted $L^2$-Hardy inequalities  on Riemannian manifolds. Inspired by \cite{Ca}, a systematic study is carried out by Berchio,  Ganguly and  Grillo \cite{BGD}, D'Ambrosio and Dipierro \cite{DD},   Kombe and \"Ozaydin \cite{KO,KO2}, Thiam \cite{T},  Yang, Su and Kong \cite{YSK}, etc. In particular, (\ref{basisHARDYFORECUL}) has been successfully generalized to complete non-compact Riemannian manifolds with non-positive
sectional curvature.

  Despite  a large quantity of contributions to the generalization of  (\ref{basisHARDYFORECUL}),  few results are available
 concerning the investigation of (\ref{distancebasisHardyforEcu}) in the category of Riemannian manifolds, not to mention  (\ref{unversHardyinequlityforEuc}).
Thus, a challenging question is to  identify  ambient manifolds as well as submanifolds on which  sharp Hardy inequalities like (\ref{unversHardyinequlityforEuc}) can be established? Until now there has been no answer  available in the literature.

The purpose of the present paper is to address all the aforementioned questions. In order to state our main results, we introduce some notations first (for details, see Section \ref{submanifold}). Let $(M,g)$ be an $m$-dimensional complete Riemannian manifold. Denote by $\mathbf{Ric}_M$ (resp., $\mathbf{K}_M$) the Ricci curvature (resp., the sectional curvature) of $M$.
Let $i:N\hookrightarrow M$ be a closed submanifold (without boundary) of codimension $k$. Particularly, $N$ is a  singleton $\{o\}$ when $k=m$.
Recall that  $N$ is said to be {\it minimal} if the mean curvature vanishes identically.  And
we define $N$  to be  {\it weakly mean convex}  in a similar way to the Euclidean case. Set $r(x):=d(N,x)$, i.e., the distance function from $N$  induced by the metric $g$.
Let  $\Omega\subset M$ be a non-empty domain (with piecewise smooth boundary or without boundary)  and set $C^\infty_0(\Omega,N):=\{u\in C^\infty_0(\Omega):u(N)=0\}$. Thus
our first main result reads as follows.

\begin{theorem}\label{keycompactHardy}
Let $(M,g)$ be an $m(\geq 2)$-dimensional complete (possibly compact) Riemannian manifold, let $i:N\hookrightarrow M$ be a  closed  submanifold of codimension $k(\geq 1)$ and let $\Omega$ be a non-empty domain in $M$.
 Suppose one of the following conditions holds:
\begin{itemize}

\item  $\mathbf{Ric}_M\geq 0$, $k=m$, $N=\{o\}\subset \Omega$;

\smallskip

\item  $\mathbf{Ric}_M\geq 0$, $k=1$ and $N=\partial\Omega$ is weakly mean convex;

\smallskip

\item  $\mathbf{K}_M\geq 0$, $1\leq k\leq m-1$, $N$ is minimal and $\Omega\cap N\neq\emptyset$.
\end{itemize}
Thus, for any $p,\beta\in \mathbb{R}$ with $1<p\neq k$ and $\beta< -k$, the following  statements are true:
 \begin{itemize}
\item[(a)]
 There always holds
 \begin{align*}
\int_{{\Omega}} |\nabla u|^pr^{p+\beta}{\dvol}_g\geq \left|\frac{\beta+ k }{p}\right|^p\int_{{\Omega}} |u|^pr^\beta {\dvol}_g, \ \forall\, u\in C^\infty_0({\Omega\backslash N}).\tag{1.5}\label{firsthardpructre}
\end{align*}
 In particular, if $N$ is compact, then (\ref{firsthardpructre}) is optimal in the following sense
 \[
\left|\frac{\beta+ k }{p}\right|^p=\inf_{u\in C^\infty_0(\Omega\backslash N)\backslash\{0\}}\frac{\int_{{\Omega}} |\nabla u|^pr^{p+\beta}{\dvol}_g}{\int_{{\Omega}} |u|^pr^\beta {\dvol}_g}.%\tag{1.2}\label{firssharpconsnta}
\]
However, this optimal constant cannot be achieved.

\smallskip

 \item[(b)]
 If $p,\beta$ satisfy  $ p+\beta>-k$, then
\begin{align*}
\int_{{\Omega}} |\nabla u|^pr^{p+\beta}{\dvol}_g\geq \left|\frac{\beta+ k }{p}\right|^p\int_{{\Omega}} |u|^pr^\beta {\dvol}_g, \ \forall\, u\in C^\infty_0(\Omega,N),\tag{1.6}\label{sendhardnonprct}
\end{align*}
In particular, if $N$ is compact, then (\ref{sendhardnonprct}) is optimal in the following sense
\[
\left|\frac{\beta+ k }{p}\right|^p=\inf_{u\in C^\infty_0(\Omega,N)\backslash\{0\}}\frac{\int_{{\Omega}} |\nabla u|^pr^{p+\beta}{\dvol}_g}{\int_{{\Omega}} |u|^pr^\beta {\dvol}_g}.\tag{1.7}\label{sharpnessoncup}
\]
However, this optimal constant cannot be achieved.

\end{itemize}
\end{theorem}

Now we give two remarks on this theorem:
\begin{itemize}

\item Theorem \ref{keycompactHardy} presents geometric descriptions of ambient manifolds as well as submanifolds on which Hardy inequalities hold. In particular,
the curvature here is {\it non-negative}.
Although this assumption is totally different from the usual one that the manifold has   non-positive curvature, it is quite natural because
Soul theorem (cf. Cheeger and Gromoll \cite{CG}) implies that every complete non-compact Riemannian manifold with non-negative sectional curvature always has a  minimal submanifold, which provides the existence of $N$ in Theorem \ref{keycompactHardy}. By
contrast, there can be no compact minimal submanifolds in   a simply-connected complete Riemannian manifold of {\it non-positive} sectional curvature, i.e.,  Cartan-Hadamard manifold (cf. Frankel \cite{Fr});

\smallskip

\item
Another distinct assumption here is that the ambient manifold $M$ can be {\it compact},  which is also natural due to the Bonnet-Myers theorem. And
 Theorem \ref{keycompactHardy} presents optimal Hardy inequalities in this case, which is genuinely new compared with the aforementioned work. Besides, Theorem \ref{keycompactHardy}/(b)
indicates that Hardy inequality can be extended to a larger space   $C^\infty_0(\Omega, N)$ than $C^\infty_0(\Omega\backslash N)$.
  In particular,  the requirement $p+\beta>-k$ is very weak; indeed, it is necessary for the local integrability of $r^{p+\beta}$ over $\Omega$ (see Lemma \ref{centerzeroinfite} below). The following example shows that $C^\infty_0(\Omega, N)$ is an ``optimal" space to study Hardy inequalities if  $M$ is compact.
\end{itemize}

\begin{example}\label{compactexampletorus}
Let $(M,g)$ be a flat torus $\mathbb{T}=\mathbb{S}^1\times \mathbb{S}^1$ and  set $N=\mathbb{S}^1\times \{o\}$, where $o$ is a fixed point in $\mathbb{S}^1$. Thus,   $\mathbf{K}_M\equiv0$ and $N$ is a minimal submanifold. Hence, Theorem \ref{keycompactHardy}/(b) yields for any $p>1$,
\[
\int_{M} |\nabla u|^p {\dvol}_g\geq \left(\frac{p-1 }{p}\right)^p\int_{M} \frac{|u|^p}{r^p} {\dvol}_g, \ \forall \,u\in C_0^\infty(M,N).
\]
On the one hand, this inequality is invalid for $C^\infty_0(M)$ because the ambient manifold $M$ is compact, in which case constant functions belong to $C^\infty_0(M)$. On the other hand, although the inequality holds for $C^\infty_0(M\backslash N)$, we have $C^\infty_0(M\backslash N)\subset C_0^\infty(M,N)=\{u\in C^\infty(M):\,u(N)=0\}$.
\end{example}

The compactness of submanifolds plays an important role in Theorem \ref{keycompactHardy}. In particular,
 even though $r(x)$ is unbounded, the sharpness still holds when $N$ is compact. Hence, Theorem \ref{keycompactHardy} covers
 (\ref{distancebasisHardyforEcu}) and Lewis, Li and Li \cite[Theorem 1.2]{LLL}.
Alternatively, for non-compact submanifolds, inspired by Barbatis et al. \cite{BFT}, we obtain the following result by requiring the finiteness of $\sup_{x\in M}r(x)$.

\begin{theorem}\label{non-compactoptimalmainthe}
Let $(M,g)$ be an $m(\geq 2)$-dimensional complete (possibly compact) Riemannian manifold, let $i:N\hookrightarrow M$ be a  closed  submanifold of codimension $k(\geq 1)$ and let $\Omega$ be a non-empty domain in $M$ with $\sup_{x\in \Omega}r(x)<+\infty$.   Suppose one of the following conditions holds:
\begin{itemize}

\item  $\mathbf{Ric}_M\geq 0$, $k=m$, $N=\{o\}\subset \Omega$;

\smallskip

\item  $\mathbf{Ric}_M\geq 0$, $k=1$ and $N=\partial\Omega$ is weakly mean convex;

\smallskip

\item  $\mathbf{K}_M\geq 0$, $1\leq k\leq m-1$, $N$ is minimal and $\Omega\cap N\neq\emptyset$.
\end{itemize} Thus,
for any $p,\beta\in \mathbb{R}$ with $1<p\neq k$ and $\beta<-k$, the following statements hold:
\begin{itemize}
\item[(a')] There always holds
 \begin{align*}
\int_{{\Omega}} |\nabla u|^pr^{p+\beta} {\dvol}_g\geq \left|\frac{\beta+ k }{p}\right|^p\int_{{\Omega}}  {|u|^p}{r^\beta} {\dvol}_g,   \ \forall\, u\in C^\infty_0({\Omega\backslash N}).\tag{1.8}\label{strongimproveHardy}
\end{align*}
In particular, (\ref{strongimproveHardy}) is optimal in the following sense
\begin{align*}
\left|\frac{\beta+ k }{p}\right|^p&=\inf_{u\in C^\infty_0(\Omega\backslash N)\backslash\{0\}}\frac{\int_{{\Omega}} |\nabla u|^pr^{p+\beta}{\dvol}_g}{\int_{{\Omega}} |u|^pr^\beta {\dvol}_g}.\tag{1.9}\label{cosntantsharpnoimpH}
\end{align*}

\item[(b')] There exists a constant $\mathcal {T}=\mathcal {T}(p,\beta,k)> 1$ such that for any $D\geq \mathcal {T}\sup_{x\in \Omega}r(x)$,
 \begin{align*}
\int_{{\Omega}} |\nabla u|^pr^{p+\beta} {\dvol}_g\geq \left|\frac{\beta+ k }{p}\right|^p\int_{{\Omega}}  {|u|^p}{r^\beta} {\dvol}_g+\frac{p-1}{2p}\left|\frac{\beta+ k }{p}\right|^{p-2}\int_\Omega  {|u|^p}{r^\beta}\log^{-2}\left( \frac{D}{r} \right){\dvol}_g,\tag{1.10}\label{strongimproveHardy2}
\end{align*}
for any $u\in C^\infty_0({\Omega\backslash N})$. In particular, (\ref{strongimproveHardy2}) is optimal in the   sense of  (\ref{cosntantsharpnoimpH}) as well as
\begin{align*}
\frac{p-1}{2p}\left|\frac{\beta+ k }{p}\right|^{p-2}&=\inf_{u\in C^\infty_0(\Omega\backslash N)\backslash\{0\}}\frac{\int_{{\Omega}} |\nabla u|^pr^{p+\beta}{\dvol}_g-\left|\frac{\beta+ k }{p}\right|^p\int_{{\Omega}} |u|^pr^\beta {\dvol}_g}{\int_\Omega  {|u|^p}{r^\beta}\log^{-2}\left( \frac{D}{r} \right){\dvol}_g}.\tag{1.11}\label{strongnessconstantH}
\end{align*}

\smallskip

\item[(c')] If  $p,\beta$  satisfy  $p+\beta>-k$, then  both (\ref{strongimproveHardy}) and  (\ref{strongimproveHardy2}) are valid for
  $u\in C^\infty_0({\Omega, N})$. In particular, these inequalities are still optimal in the following sense
\begin{align*}
\left|\frac{\beta+ k }{p}\right|^p&=\inf_{u\in C^\infty_0(\Omega, N)\backslash\{0\}}\frac{\int_{{\Omega}} |\nabla u|^pr^{p+\beta}{\dvol}_g}{\int_{{\Omega}} |u|^pr^\beta {\dvol}_g},\\%\tag{1.2}\label{firssharpconsnta}\\
\frac{p-1}{2p}\left|\frac{\beta+ k }{p}\right|^{p-2}&=\inf_{u\in C^\infty_0(\Omega, N)\backslash\{0\}}\frac{\int_{{\Omega}} |\nabla u|^pr^{p+\beta}{\dvol}_g-\left|\frac{\beta+ k }{p}\right|^p\int_{{\Omega}} |u|^pr^\beta {\dvol}_g}{\int_\Omega  {|u|^p}{r^\beta}\log^{-2}\left( \frac{D}{r} \right){\dvol}_g}.
\end{align*}

\end{itemize}

\end{theorem}

Theorem \ref{non-compactoptimalmainthe}   presents a weighted-version as well as a compact-ambient-space-version of   Barbatis et al. \cite[Theorem A]{BFT}  in the context of Riemannian manifolds.
  In view of Theorem \ref{non-compactoptimalmainthe}/(c'), it is quite surprising that both of these constants remain best for $C^\infty_0(\Omega,N)$, which together with Theorem \ref{keycompactHardy}/(b)  implies that $C^\infty_0({\Omega, N})$ is an excellent candidate for studying Hardy inequalities.

We remark that recently the special case $k=m$ (i.e., $N$ is a  singleton)
has also been  studied by the third author on Finsler manifolds of non-negative  weighted Ricci curvature \cite{Z3} as well as by Meng, Wang and the third author on Riemannian metric measure manifolds of non-negative weighted Ricci curvature \cite{MWZ}, in which case the comparison theory in the context of polar coordinates plays an important role. Alternatively, in the present paper we mainly take advantage of {\it Fermi coordinates} and the corresponding comparison theorems (see Section \ref{submanifold} below), which provide an efficient way of dealing with the  case $k<m$. Moreover, in order to investigate the sharpness of the Hardy inequality   over the larger space $C^\infty_0(\Omega,N)$, we extend the theory of  Sobolev capacity in a weighted Riemannian setting (see Appendix \ref{Soblevspace} below).

At the end of this section, we discuss briefly the influence of curvature on the validity of Hardy inequalities in Theorems \ref{keycompactHardy}$\&$\ref{non-compactoptimalmainthe}.
It is well-known that Hardy inequalities are the limiting cases of Caffarelli-Kohn-Nirenberg (CKN) inequalities. For a  manifold of  non-negative Ricci curvature,
   the validity of   CKN inequalities  is equivalent to that the manifold is {\it trivial/flat}, i.e., the sectional curvature is identically zero (cf. Carmo and Xia \cite{CW}, Krist\'aly \cite{K} and Xia \cite{X}, etc.). However, this is another story for Hardy inequalities.
For instance, although a $2$-dimensional unit sphere $\mathbb{S}^2$ has strictly positive Ricci/sectional curvature $1$,  one can derive a sphere version of (\ref{distancebasisHardyforEcu}) from a direct calculation (see  Section \ref{examplesection} below), i.e.,
   \[
\int_{\mathbb{S}^2_+}|\nabla u|^p\dvol_g\geq \left( \frac{p-1}{p} \right)^p\int_{\mathbb{S}^2_+}  \frac{|u|^p}{{r}^p} \dvol_g, \ \forall\,u\in C_0^\infty(\mathbb{S}^2_+), \ \forall\,p>1,
\]
where $\mathbb{S}^2_+$ is the upper hemisphere of $\mathbb{S}^2$ and $r(x):=d(\partial \mathbb{S}^2_+,x)$ is the distance  between $\partial \mathbb{S}^2_+$ and $x$ on $\mathbb{S}^2$. This example indicates that Theorems \ref{keycompactHardy}$\&$\ref{non-compactoptimalmainthe} are correct as well as nontrivial.
On the other hand, the subtle differences do occur between the flat case and the curved case.  For example, consider the non-weighted Hardy inequalities (i.e., $p+\beta=0$) in Theorems \ref{keycompactHardy}$\&$\ref{non-compactoptimalmainthe}. Usually we have to require $p>k$ because of the impact of the Ricci/sectional curvature and the mean curvature. However, if the ambient manifold is flat while the submanifold is minimal, then  this requirement can be eliminated, which covers (\ref{basisHARDYFORECUL}). See Theorem \ref{flatcaseThe1112} below for the
formulation of such statements.

\medskip

The paper is organized as follows. Section \ref{submanifold} is devoted to  Fermi coordinates and  comparison theorems in submanifold geometry. In Section \ref{basistheorem} we prove Theorem \ref{keycompactHardy} and study the Hardy inequalities with   logarithmic weights.
The proof of Theorem \ref{non-compactoptimalmainthe} is given in Section \ref{non-compacthary}.   In Section \ref{flathardy} we deal with the flat case.  Some interesting examples are presented in Section \ref{examplesection}, which show the validity of Theorems \ref{keycompactHardy}, \ref{non-compactoptimalmainthe} and \ref{flatcaseThe1112}.
 We devote Appendix \ref{Soblevspace} to weighted Sobolev spaces over Riemannian manifolds, which provides  some necessary tools used throughout the previous
sections.

\section{Riemannian submanifolds}\label{submanifold}
Throughout this paper, we always assume that $(M,g)$ is an $m(\geq 2)$-dimensional complete (possibly compact) Riemannian manifold and $i:N\hookrightarrow M$ is an $n$-dimensional (connected) {\it closed submanifold} (cf. Chern et al. \cite[p.\,21]{CCL})  without boundary, where $0\leq n\leq m-1$. Briefly speaking, $(i,N)$ is a regular submanifold of $M$ and $i(N)$ is a closed set of $M$.
  In  particular, $N$ is a pole (i.e., singleton)  if $n=0$.  For simplicity of presentation, we prefer to use the dimension $n$ of $N$ rather than the codimension $k=m-n$.

\subsection{Fermi coordinates}

We  recall some definitions and properties of submanifold geometry, especially Fermi coordinates. See Chavel \cite[Section 3.6, Section 7.3]{IC} and Heintze-Karcher \cite{HK} for  details.

Let $(M,g)$ and $N$ be the ambient manifold and the submanifold  described above.
Let $\mathcal {V}N$ (resp.,  $\mathcal {V}SN$) denote the normal bundle (resp., the unit normal bundle) of $N$ in $M$. The fiber of $\mathcal {V}N$ (resp., $\mathcal {V}SN$) at $x\in N$ is denoted by $\mathcal {V}_xN$ (resp., $\mathcal {V}S_xN$). %Let $c_{m-n-1}$ denotes the volume of the unit Euclidean $(m-n-1)$-sphere $\mathbb{S}^{m-n-1}$.
%The rules that govern our index gymnastics are as follows: $i,j$ run from $1$ to $m$; $\alpha,\beta$ run from $1$ to $n$; $\mathfrak{g}, \mathfrak{h}$ run from $n+1$ to $m-1$ and $\mathbbm{a},\mathbbm{b}$ run from $1$ to $m-1$.

The {\it second fundamental
form} $\mathfrak{B}$ of $N$ in $M$, for every $x\in N$,  is a vector-valued symmetric form $\mathfrak{B}:T_xN\times T_xN\rightarrow \mathcal {V}_xN$, given by
\[
\mathfrak{B}(X,Y):=(\nabla_X\overline{Y})^\bot,
\]
where $\overline{Y}$ is any extension of $Y$ to a tangent vector field on $N$, $\nabla$ is the
Levi-Civita connection of $(M,g)$, and the superscript $\bot$ denotes the projection onto $\mathcal {V}_xN$.
Given $\mathbf{n}\in \mathcal {V}S_xN$, the {\it Weingarten map} $\mathfrak{A}^{\mathbf{n}}:T_xN\rightarrow T_xN$ is defined by
\[
\mathfrak{A}^{\mathbf{n}}(X):=-(\nabla_X\overline{\mathbf{n}})^\top,
\]
where $\overline{\mathbf{n}}$ is any extension of $\mathbf{n}$ to a normal vector field on $N$, and the superscript
$\top$ denotes the  projection onto $T_xN$.
The {\it mean curvature} is defined by
\[
H:=\tr_{i^*g}\mathfrak{B}.
\]
For any $\mathbf{n}\in \mathcal {V}S_xN$, there holds $\tr\mathfrak{A}^{\mathbf{n}}=g(H,\mathbf{n})$.  And $N$ is said to be {\it a minimal submanifold}  if $H$ vanishes identically.
Inspired by Lewis, Li and Li\cite{LLL}, we introduce the following definition.
\begin{definition} Let $i:N\hookrightarrow M$ be  an $(m-1)$-dimensional submanifold.
 $N$ is said to be {\it  weakly mean convex}, if there exists a domain (i.e., a connected open subset) $\Omega\subset M$ such that $\partial \Omega=N$ and
$\tr \mathfrak{A}^\mathbf{n}\geq 0$
 for any inward normal vector $\mathbf{n}$  with respect to $\Omega$.
\end{definition}
\begin{remark}
The definition of mean curvature in Chavel \cite{IC} is  different from the one in  Lewis et al. \cite{LLL} by a factor $-1$. Hence, here the weakly mean convexity needs the non-negativity of the mean curvature with respect to the {\it inward} normal vectors.
\end{remark}

\begin{remark}\label{otherdefinaboutB}
A  totally geodesic submanifold  is always   minimal  while the smooth boundary of a convex domain is always weakly mean convex.
\end{remark}

%\begin{example}
%Let $M=\mathbb{R}^n$ and let $\Omega$ be a unit Euclidean ball $\mathbb{B}^n$.  Choose a frame field $\{e_\alpha,\mathbf{n}\}$ along $\partial\Omega=\mathbb{S}^{n-1}$, where $\{e_\alpha\}_{\alpha=1}^{n-1}$ is an arbitrary orthonormal basis of $T\mathbb{S}^{n-1}$ while %$\mathbf{n}$ is the unit inward normal vector field of $\mathbb{S}^n$. A direct calculation yields $\tr\mathfrak{A}^{\mathbf{n}}=(n-1)$ and $g(\mathfrak{B},\mathbf{n})=\id_{(n-1)\times(n-1)}$.
%\end{example}

The {\it exponential map  of normal bundle $\Exp:\mathcal {V}N\rightarrow M$}  is defined as
 \[
 \Exp(\mathbf{n}):=\exp_{\pi(\mathbf{n})}(\mathbf{n}),
 \]
 where
  $\pi:\mathcal {V}N\rightarrow N$ is the bundle projection and $\exp_{\pi(\mathbf{n})}:T_{\pi(\mathbf{n})}M\rightarrow M$ is the exponential map.
  According to Gray \cite{G} and O'Neill \cite{ON}, given $x\in N$, there exists a neighborhood  $U\subset N$ of $x$ and $\epsilon_x>0$ such that $\Exp|_{\pi^{-1}[U]\cap \mathbb{B}(U,\epsilon_x)}$ is a diffeomorphism of ${\pi^{-1}[U]\cap \mathbb{B}(U,\epsilon_x)}$, where $\mathbb{B}(U,\epsilon_x):=\{y\in \cup_{z\in U}T_zM:\, 0<|y|<\epsilon_x  \}$.
   In particular, if $N$ is compact,
there exists a small $\epsilon>0$ such that $\Exp_{*t\mathbf{n}}$ is nonsingular for all $t\in [0, \epsilon)$ and $\mathbf{n}\in \mathcal {V}SN$.

The so-called {\it Fermi coordinates} $(t,\mathbf{n})\in (0,+\infty)\times \mathcal {V}SN$ are defined by the map $E:[0,+\infty)\times \mathcal {V}SN\rightarrow M$, $(t,\mathbf{n})\mapsto\Exp(t\mathbf{n})$ (also see Lemma \ref{Fermiseconlemma}/(iii) below).
In particular, we have the following fact (cf. Chavel \cite[p.133]{IC}).

\begin{lemma}\label{Fermifirstlemma} Let $E:[0,\infty)\times \mathcal {V}SN\rightarrow M$ be defined as $(t,\mathbf{n})\mapsto\Exp(t\mathbf{n})$. If $E(t,\mathbf{n})$ is well-defined,  then
\begin{itemize}
\item[(i)] $\gamma_\mathbf{n}(t):=E(t,\mathbf{n})$, $t\geq 0$ is a normal geodesic from $N$ with the initial velocity $\mathbf{n}$;

\smallskip

	\item[(ii)] $\frac{\partial}{\partial t}=E_{*(t,\mathbf{n})}\mathbf{n}=\dot{\gamma}_\mathbf{n}(t)$, i.e., the tangent vector of  $\gamma_\mathbf{n}(t)$ at $t$;

\smallskip

\item[(iii)] given $X\in T_{\pi(\mathbf{n})}N$, $J_X(t):=E_{*(t,\mathbf{n})}X$ is a Jacobi field along $\gamma_\mathbf{n}(t)$, orthogonal to $\gamma_{\mathbf{n}}$, with
\[
J_X(0)=X,\quad (\nabla_tJ_X)(0)=-\mathfrak{A}^{\mathbf{n}}(X);
\]

 \item[(iv)] given $y\in \mathcal {V}_{\pi(\mathbf{n})}N\cap \mathbf{n}^\bot$, $J_y(t):=E_{*(t,\mathbf{n})}y$ is a Jacobi field along $\gamma_\mathbf{n}(t)$,  orthogonal to $\gamma_{\mathbf{n}}$,  with
 \[
 J_y(0)=0,\quad (\nabla_tJ_y)(0)=y,
 \]
where $\mathbf{n}^\bot:=\{y\in T_{\pi(\mathbf{n})}M:\,g(\mathbf{n},y)=0 \}$.

\end{itemize}
\end{lemma}

Let $R$ denote the curvature tensor of $(M,g)$, i.e.,
\[
R(X,Y)Z=\nabla_X\nabla_YZ-\nabla_Y\nabla_XZ-\nabla_{[X,Y]}Z,
\]
where $X,Y,Z$ are three smooth vector fields on $M$. Given $\mathbf{n}\in \mathcal {V}S_xN$,
denote by $\tau_{t;\mathbf{n}}$ the parallel transport along
$\gamma_{\mathbf{n}}$ from $T_{\gamma_\mathbf{n}(0)}M$ to $T_{\gamma_\mathbf{n}(t)}M$  for all $t\geq 0$. Set
$R:=R\left(\,\cdot\,,\dot{\gamma}_\mathbf{n}(t)\right)\dot{\gamma}_\mathbf{n}(t)$ and
\[
\mathcal {R}(t,\mathbf{n} ):=\tau^{-1}_{t;\mathbf{n}}\circ R \circ \tau_{t;\mathbf{n}}:
\mathbf{n}^{\bot}\rightarrow \mathbf{n}^{\bot}.\]
Let $\mathcal {A}(t,\mathbf{n})$ be the solution of the matrix (or linear
transformation) ordinary differential equation on $\mathbf{n}^\bot$:
\[
\left \{
\begin{array}{lll}
\mathcal {A}{''}(t,\mathbf{n})+\mathcal {R}(t,\mathbf{n})\,\mathcal {A}(t,\mathbf{n})=0,\\
\\
\mathcal {A}(0,\mathbf{n})|_{T_xN}=\id|_{T_xN},\ \mathcal {A}'(0,\mathbf{n})|_{T_xN}=-\mathfrak{A}^{\mathbf{n}},\tag{2.1}\label{defineJacbi}\\
\\
 \mathcal {A}(0,\mathbf{n})|_{\mathcal {V}_xN\cap \mathbf{n}^\bot}=0,\ \mathcal {A}'(0,\mathbf{n})|_{\mathcal {V}_xN\cap \mathbf{n}^\bot}=\id|_{\mathcal {V}_xN\cap \mathbf{n}^\bot},
\end{array}
\right.
\]
where $\mathcal {A}'=\frac{d}{d t}\mathcal {A}$ and $\mathcal {A}''=\frac{d^2}{d t^2}\mathcal {A}$. Thus, Lemma \ref{Fermifirstlemma} indicates
\[
\left \{
\begin{array}{lll}
\frac{\partial}{\partial t}=\dot{\gamma}_{\mathbf{n}}(t)=\tau_{t;\mathbf{n}}\mathbf{n},\\
\\
J_X(t)=E_{*(t,\mathbf{n})}X=\tau_{t;\mathbf{n}}\mathcal {A}(t,\mathbf{n})X, \ \forall\,X\in T_xN,\\
\\
J_y(t)=E_{*(t,\mathbf{n})}y=\tau_{t;\mathbf{n}}\mathcal {A}(t,\mathbf{n})y,\ \forall\,y\in \mathcal {V}_xN\cap \mathbf{n}^\bot.
\end{array}
\right.
\]
Moreover,   (\ref{defineJacbi}) yields
\begin{align*}
\underset{t\rightarrow 0^+}{\lim}\frac{\det \mathcal {A}(t,\mathbf{n})}{t^{m-n-1}}=1,\ \frac{(\det \mathcal {A}(t,\mathbf{n}))' }{\det \mathcal {A}(t,\mathbf{n})}=\frac{(m-n-1)}t- \tr\mathfrak{A}^{\mathbf{n}}+O(t).\tag{2.2}\label{smallsestimatedatA}
\end{align*}

Given $\mathbf{n}\in \mathcal {V}SN$,   the {\it distance to the focal cut point}
of $N$ along $\gamma_\mathbf{n}$ is defined as
\[
c_{\mathcal {V}}(\mathbf{n}):=\sup\{t>0:\,d(N,\gamma_\mathbf{n}(t))=t\}.
\]
Thus $\det\mathcal {A}(t,\mathbf{n})>0$ for any $t\in (0,c_\mathcal {V}(\mathbf{n}))$. In addition, one can introduce
\begin{align*}
&\mathcal {V}C(N):=\{ c_{\mathcal {V}}(\mathbf{n})\mathbf{n}:\,c_{\mathcal {V}}(\mathbf{n})<+\infty,\mathbf{n}\in \mathcal {V}SN\},&\mathcal {V}\mathscr{C}(N):=\Exp\mathcal {V}C(N),\\
&\mathcal {V}D(N):=\left\{t\mathbf{n}:\,0< t<c_{\mathcal {V}}(\mathbf{n}),\, \mathbf{n}\in \mathcal {V}SN\right\}, &\mathcal {V}\mathscr{D}(N):=\Exp\mathcal {V}D(N).
\end{align*}
According to Chavel \cite[p.\,105ff, p.\,134]{IC}, the following result holds.
\begin{lemma}\label{Fermiseconlemma} Let $(M,g)$ be an $m$-dimensional complete Riemannian manifold and let $i:N\hookrightarrow M$ be an $n$-dimensional closed submanifold. Then the following statements  are true:
\begin{itemize}
\item[(i)] $\gamma_\mathbf{n}(t)$, $t\in [0,c_{\mathcal {V}}(\mathbf{n})]$ is a minimal geodesic from $N$ to $\gamma_{\mathbf{n}}(c_{\mathcal {V}}(\mathbf{n}))$. In particular, for any $0<t<c_{\mathcal {V}}(\mathbf{n})$, $\gamma_\mathbf{n}$ is the unique minimal geodesic from $N$ to $\gamma_{\mathbf{n}}(t)$.

    \smallskip

    \item[(ii)] $c_{\mathcal {V}}:\mathcal {V}SN\rightarrow (0,+\infty]$ is a continuous function. Hence, $\mathcal {V}\mathscr{C}(N)$ is a set of Lebesgue measure $0$.

     \smallskip

    \item[(iii)] The map $\Exp:\mathcal {V}D(N)\rightarrow\mathcal {V}\mathscr{D}(N)$ is a diffeomorphism. Particularly, $\mathcal {V}\mathscr{D}(N)=M\backslash \left(\mathcal {V}\mathscr{C}(N) \cup N\right)$.
\end{itemize}
\end{lemma}

In the sequel, we recall the further properties of $\det\mathcal {A}(t,\mathbf{n})$. Firstly,
the proof of the Heintze-Karcher theorem \cite{HK} (also see Chavel \cite[Theorem 7.5]{IC}) furnishes the following result.
\begin{theorem}[Heintze-Karcher theorem]\label{importantlemmahk} Let $(M,g)$ be an $m$-dimensional complete Riemannian manifold with $\mathbf{K}_M\geq K$ and let $i:N\hookrightarrow M$ be an $n(\geq 1)$-dimensional  closed submanifold.
For any $\mathbf{n}\in \mathcal {V}SN$, we have
\[
\frac{(\det \mathcal {A}(t,\mathbf{n}))'}{\det \mathcal {A}(t,\mathbf{n})}\leq \frac{(\det \mathcal {A}_K(t))'}{\det \mathcal {A}_K(t)},\text{ for }0<t<c_\mathcal {V}(\mathbf{n}),
\]
where
\[
\det \mathcal {A}_K(t):=\mathfrak{s}^{m-n-1}_K(t)\cdot \overset{n}{\underset{\alpha=1}{\prod}}\left(  \mathfrak{s}'_K(t)-\lambda_\alpha \,\mathfrak{s}_K(t) \right),
\]
and $\{\lambda_\alpha\}_{\alpha=1}^n$ are the eigenvalues of $\mathfrak{A}^{\mathbf{n}}$ with respect to some orthonormal basis of $T_{\pi(\mathbf{n})}N$. Moreover,
\[
\det \mathcal {A}(t,\mathbf{n})\leq \left[ \mathfrak{s}'_K(t)-\frac{\tr \mathfrak{A}^\mathbf{n}}{n}\mathfrak{s}_K(t) \right]^n \mathfrak{s}^{m-n-1}_K(t), \ \text{ for }t\in [0,c_{\mathcal {V}}(\mathbf{n})],
\]
where $\mathfrak{s}_K(t)$  is the unique solution to $\mathfrak{s}''_K(t)+K \mathfrak{s}_K(t)=0$ with $\mathfrak{s}_K(0)=0$ and $\mathfrak{s}'_K(0)=1$.
\end{theorem}

Secondly, Chavel \cite[Theorem 3.14]{IC} together with the argument in \cite[p.\,121f]{IC} furnishes the following result.
\begin{theorem}\label{hypernonlaplac}
Let $(M,g)$ be an $m$-dimensional complete Riemannian manifold with $\mathbf{Ric}_M\geq (m-1)K$ and let $i:N\hookrightarrow M$ be an $(m-1)$-dimensional closed  submanifold.   Given $\mathbf{n}\in \mathcal {V}SN$, if  $\tr  \mathfrak{A}^{\mathbf{n}}\geq (m-1)\lambda$,
then
\[
\frac{(\det \mathcal {A}(t,\mathbf{n}))'}{\det \mathcal {A}(t,\mathbf{n})}\leq (m-1)\left[\frac{\mathfrak{s}''_K(t)-\lambda \mathfrak{s}'_K(t)}{\mathfrak{s}'_K(t)-\lambda \mathfrak{s}_K(t)}\right], \ \det \mathcal {A}(t,\mathbf{n})\leq \left(  \mathfrak{s}'_K(t)-\lambda\mathfrak{s}_K(t) \right)^{m-1},  \ \forall\, t\in(0,c_\mathcal {V}(\mathbf{n})).
\]
\end{theorem}

\subsection{Estimates of volume and Laplacian}
In this subsection, we present several estimates of volume and the Laplacian by Fermi coordinates. In what follows, $\mathbb{S}^n$ denotes an $n$-dimensional unit sphere in $\mathbb{R}^{n+1}$ and $c_{n}:=\vol(\mathbb{S}^n)$.

 Let $(u^\alpha)$ be a local coordinate system of $N$ and let $(\theta^\mathfrak{g})$ be a local coordinate system of $\mathbb{S}^{m-n-1}\approx\mathcal {V}S_xN$. Thus, $(t,\mathbf{n})=(t,u^\alpha,\theta^\mathfrak{g})$ is a Fermi coordinate system of $M$.
For $x\in N$, define $e_\mathbbm{a}\in \mathbf{n}^\bot\cap T_xM$ as
\[
e_\mathbbm{a}:=\left \{
\begin{array}{lll}
\frac{\partial}{\partial u^\alpha},& \text{ for }\mathbbm{a}=\alpha, \\
\tag{2.3}\label{defofea}\\
\frac{\partial}{\partial\theta^\mathfrak{g}},& \text{ for }\mathbbm{a}=\mathfrak{g}.
\end{array}
\right.
\]
Let $J_\mathbbm{a}(t):=\tau_{t;\mathbf{n}}\mathcal {A}(t,\mathbf{n})e_{\mathbbm{a}}$. Since $g(e_\alpha,e_\mathfrak{g})=0$, a direct calculation yields
\begin{align*}
\det\left[g(J_\mathbbm{a}(t),J_\mathbbm{b}(t))\right]=(\det\mathcal {A})^2\cdot\det g(e_\alpha,e_\beta)\cdot\det g(e_\mathfrak{g},e_\mathfrak{h}).
\end{align*}
In view of Lemma \ref{Fermifirstlemma}, the Riemannian measure $\dvol_g$ of $(M,g)$ can be written as
\begin{align*}
{\dvol}_g(t,\mathbf{n})&=\det \mathcal {A}(t,\mathbf{n}) dt \cdot \sqrt{\det g(e_\alpha,e_\beta)} \,du^1\cdots du^{n}\cdot \sqrt{\det g(e_\mathfrak{g},e_\mathfrak{h})} \,d\theta^1\cdots d\theta^{m-n-1}\\
&=\det \mathcal {A}(t,\mathbf{n}) dt \cdot {\dvol}_{i^*g}(x) \cdot d\nu_x(\mathbf{n}),\tag{2.4}\label{fermicoordinvolumform}
\end{align*}
where ${\dvol}_{i^*g}$ is the induced Riemannian measure of $(N,i^*g)$ and $d\nu_x$ is the Riemannian measure of $\mathcal {V}S_xN$.
Furthermore, by Lemma
\ref{Fermiseconlemma}, for $f\in L^1(M)$, we have
\[
\int_M f {\dvol}_g=\int_{N}{\dvol}_{i^*g}(x)\int_{\mathcal {V}S_xN}d\nu_x(\mathbf{n})\int_0^{c_\mathcal {V}(\mathbf{n})}f(\Exp(t\mathbf{n}))\, \det \mathcal {A}(t,\mathbf{n})dt.\tag{2.5}\label{2.0'fermico}
\]

Let $d:M\times M\rightarrow [0,\infty)$ denote the distance function induced by the Riemannian metric $g$. The {\it distance function from $N$} is defined by
\[
r(x):=d(N,x),\ \forall\,x\in M.
\]
 The gradient vector field $(\nabla r)(x)$ has unit length when $x\in \mathcal {V}\mathscr{D}(N)$. More precisely,
 for any $x=(t,\mathbf{n})\in \mathcal {V}\mathscr{D}(N)$,  Lemmas \ref{Fermifirstlemma} and \ref{Fermiseconlemma} yield
 \[
 r(x)=t,\ (\nabla r)(x)=\dot{\gamma}_{\mathbf{n}}(t)=\left.\frac{\partial}{\partial t}\right|_{(t,\mathbf{n})},\tag{2.6}\label{distanceandt}
 \]
 where $\gamma_\mathbf{n}(s)$, $0\leq s\leq d(N,x)$, is the unique minimal geodesic from $N$ to $x$ with the initial velocity $\mathbf{n}$.

Throughout this paper, we always use  $T_s$ to denote the {\it $s$-tuber neighborhood of $N$}, i.e.,
\[
T_s:=\{x\in M:\ r(x)<s\}.
\]
\begin{lemma}\label{centerzeroinfite}Let $(M,g)$ be an $m$-dimensional complete Riemannian manifold and let $i:N\hookrightarrow M$ be an $n$-dimensional closed submanifold. For any bounded domain $\Omega$ with $\overline{\Omega} \cap N\neq \emptyset$, we have
\begin{itemize}
\item[(i)] if $l<m-n$, then $\ds\ds\lim_{\epsilon\rightarrow 0^+}\int_{\Omega \cap T_\epsilon} r^{-l} {\dvol}_g=0$;

 \smallskip

\item[(ii)] if $l\geq m-n$, then $\ds\ds \lim_{\epsilon\rightarrow 0^+}\int_{\Omega\backslash T_\epsilon} r^{-l} {\dvol}_g=+\infty$.

\end{itemize}
\end{lemma}
\begin{proof} Without loss of generality, we may assume that $N\subset \overline{\Omega}$ and hence, $N$ is compact. Thus, Lemma \ref{Fermiseconlemma}/(ii) together with (\ref{smallsestimatedatA}) yields
a small $\eta\in (0,  \min_{\mathbf{n}\in \mathcal {V}SN}c_\mathcal {V}(\mathbf{n}))$ such that
\[
\frac{t^{m-n-1}}2\leq {\det \mathcal {A}(t,\mathbf{n})}\leq 2{t^{m-n-1}},\ \forall (t,\mathbf{n})\in(0,\eta)\times \mathcal {V}SN.\tag{2.7}\label{estimtevoluem}
\]
Provided $l<m-n$, it follows from
 (\ref{2.0'fermico})  that
\begin{align*}
 \int_{\Omega\cap T_\epsilon} r^{-l} {\dvol}_g\leq 2c_{m-n-1}\vol_{i^*g}(N) \int^{\epsilon}_0 t^{m-n-1-l}dt \rightarrow 0 \text{ as }\epsilon\rightarrow 0^+,
\end{align*}
which implies Statement (i).
Alternatively, if $l\geq m-n$,
then (\ref{2.0'fermico}) together with (\ref{estimtevoluem}) furnishes
\begin{align*}
\int_{\Omega\backslash T_\epsilon} r^{-l} {\dvol}_g\geq \frac12c_{m-n-1}\vol_{i^*g}(N)\int_\epsilon^{\eta} t^{m-n-l-1}dt\rightarrow +\infty \text{ as }\epsilon\rightarrow 0^+.
\end{align*}
Consequently, Statement (ii) follows.
\end{proof}

A smooth measure ${\dmu}$ can be viewed as an $m$-form on $M$. Hence, given any vector field $X$ on $M$, one can define an $(m-1)$-form $X\rfloor {\dmu}$ by
\[
X\rfloor {\dmu}\,(Y_1,\ldots,Y_{m-1}):={\dmu}(X,Y_1,\ldots,Y_{m-1}).\tag{2.8}\label{innerexterproduct}
\]
In particular, $d(X\rfloor \dvol_g)=\di(X) \dvol_g$.
Thus, it follows from  (\ref{distanceandt}) and (\ref{fermicoordinvolumform}) that
\begin{align*}
\Delta r \,{\dvol}_g=\Delta t \,{\dvol}_g=\left(\di \circ \nabla t \right)\,{\dvol}_g=d\left( \nabla t\rfloor  {\dvol}_g\right)=\frac{\left( \det \mathcal {A}(t,\mathbf{n}) \right)'}{\det \mathcal {A}(t,\mathbf{n}) }{\dvol}_g,
\end{align*}
which implies
\[
 (\Delta r)|_{(t,\mathbf{n})}=\Delta t=\frac{\left( \det \mathcal {A}(t,\mathbf{n}) \right)'}{\det \mathcal {A}(t,\mathbf{n}) }, \text{ for }0<t<c_\mathcal {V}(\mathbf{n}).\tag{2.9}\label{lapacianmeaning}
\]
Furthermore, (\ref{lapacianmeaning}) together with (\ref{smallsestimatedatA}) yields
\[
\left[r\Delta r+1-(m-n)\right]|_{(t,\mathbf{n})}= -t \cdot\tr \mathfrak{A}^{\mathbf{n}}+o(t).\tag{2.10}\label{smallestimatelaplace}
\]

Inspired by  Lewis et al. \cite{LLL}, we use the following  inequality to estimate $\Delta r$.

\begin{lemma}[Newton's inequality \cite{Ne}]\label{newtonlemma} Let $\lambda=(\lambda_1,\ldots,\lambda_n)$ be a vector with $\lambda_i>0$ for all $1\leq i\leq n$. Denote by $\sigma_s(\lambda)$  the $s$-th elementary symmetric function of the vector $\lambda$, i.e.,
\[
\sigma_s(\lambda)=\sum_{1\leq i_1<\cdots<i_s\leq n}\lambda_{i_1}\cdots\lambda_{i_s}.
\]
Then
\[
\frac{\sigma_{n-1}(\lambda)}{\sigma_{n}(\lambda)}\geq \cdots\geq c(n,s)\frac{\sigma_{s-1}(\lambda)}{\sigma_{s}(\lambda)}\geq \cdots \geq n^2\frac{1}{\sigma_1(\lambda)},
\]
where $c(n,s)=\frac{n(n-s+1)}{s}$. The above equalities hold if and only if $\lambda_1=\cdots=\lambda_n$.
\end{lemma}

\begin{theorem}\label{huperseacedata}Let $(M,g)$ be an $m$-dimensional complete Riemannian manifold with $\mathbf{K}_M\geq 0$ and let $i:N\hookrightarrow M$ be an $n(\geq 1)$-dimensional  minimal  closed submanifold. Then
$ \Delta r\leq \frac{m-n-1}{r}  \text{ a.e. in }M$.

\end{theorem}
\begin{proof} Let $(t,\mathbf{n})$ denote  Fermi coordinates. Given  $\mathbf{n}\in \mathcal {V}SN$,
Theorem \ref{importantlemmahk} together with (\ref{lapacianmeaning}) yields
\begin{align*}
r\Delta r=t\Delta t\leq t\left[\frac{m-n-1}{t}+\overset{n}{\underset{\alpha=1}{\sum}}\left(\frac{-\lambda_\alpha}{1-t\lambda_\alpha}\right)\right],\ \forall\,t\in(0,c_\mathcal {V}(\mathbf{n})),\tag{2.11}\label{middellapalcian}
\end{align*}
where  $\{\lambda_\alpha\}_{\alpha=1}^n$ are the eigenvalues of $\mathfrak{A}^{\mathbf{n}}$ with respect to some orthonormal basis of $T_{\pi(\mathbf{n})}N$.
Obviously, for each $\alpha$,
$\xi_\alpha:=1-t\lambda_\alpha>0$ for $t\in(0,c_\mathcal {V}(\mathbf{n}))$, otherwise there would be   $\tilde{t}\in (0,c_\mathcal {V}(\mathbf{n}))$ and $\tilde{\alpha}\in\{1,\ldots,n\}$ such that $1-\tilde{t}\,\lambda_{\tilde{\alpha}}=0$, which together with  Theorem \ref{importantlemmahk} yields
 $\det \mathcal {A}(\tilde{t},\mathbf{n})\leq \det \mathcal {A}_0(\tilde{t})=0$, which is a contradiction.
Since $\sigma_1(\lambda)=\tr{\mathfrak{A}^{\mathbf{n}}}=0$, Lemma \ref{newtonlemma} yields
\begin{align*}
\overset{n}{\underset{\alpha=1}{\sum}}\left(\frac{t\lambda_\alpha}{1-t\lambda_\alpha}\right)=\overset{n}{\underset{\alpha=1}{\sum}}\frac{1}{\xi_\alpha}-n=\frac{\sigma_{n-1}(\xi)}{\sigma_{n}(\xi)}-n\geq \frac{n^2}{\sigma_1(\xi)}-n=\frac{tn\,{\tr \mathfrak{A}^{\mathbf{n}}}}{n-t\,{\tr \mathfrak{A}^{\mathbf{n}}}}=0,
\end{align*}
which together with (\ref{middellapalcian}) concludes the proof.
\end{proof}

%It also should be remarked that if $n=0$ (i.e., $N$ is a pole $o$), then Fermi coordinate system degenerate into polar coordinate system about $o$, in which case Lemma \ref{huperseacedata} can be replaced by Bishop's Laplacian comparison theorem (cf.  Chavel \cite[Theorem 3.8]{IC}).

The following corollary is a direct consequence of Theorem \ref{hypernonlaplac} combined with (\ref{lapacianmeaning}).
\begin{corollary}\label{inlalpace}
Let $(M,g)$ be a complete Riemannian manifold with $\mathbf{Ric}_M\geq 0$ and let $\Omega\subset M$ be a domain  with non-empty  piecewise smooth weakly mean convex boundary $N$. Then $\Delta r\leq 0$ a.e. in $\Omega$.
\end{corollary}

\begin{remark}\label{zerodimensinocase}
If
 $n=0$ (i.e., $N$ is a pole), then  Fermi coordinates are exactly polar coordinates while $c_{\mathcal {V}}(\mathbf{n})$ is  the cut value of $\mathbf{n}$, in which case  Lemmas \ref{Fermifirstlemma} and \ref{Fermiseconlemma} (and hence, (\ref{defineJacbi})-(\ref{smallsestimatedatA})) remain valid, whereas Theorems \ref{importantlemmahk} and \ref{huperseacedata} are replaced by the Bishop-Gromov comparison theorem and the Laplacian comparison theorem, respectively (cf. Chavel \cite[Theorem 3.8]{IC} and Chow et al. \cite[Theorem 1.128]{CLN}).
\end{remark}

\section{Sharp Hardy inequalities via compact submanifolds}\label{basistheorem}

\subsection{Theory of $p$-Laplacian}

For simplicity of presentation, we introduce some  conventions and assumptions, which will be used throughout this section.

 A connected non-empty open subset $\Omega\subset M$  (with piecewise smooth boundary or without boundary) is said to be  a {\it natural domain (with respect to $N$)} if   one of the following conditions holds:

\ \ \ \ \ (a)  $n=m-1$ and $N=\partial\Omega$; \ \ \ \ \ (b) $0\leq n\leq m-1$ and $N\cap\Omega\neq\emptyset$.

\noindent Given a natural domain $\Omega$, set $\Omega_N:=\Omega\backslash N$. Furthermore, we introduce the following assumption.
\begin{assumption}\label{assufreecur1}
Let $(M,g)$ be an $m$-dimensional complete Riemannian manifold, let $i:N\hookrightarrow M$ be an $n$-dimensional closed submanifold and let $\Omega$ be a natural domain in $M$.
 \end{assumption}

 Proceeding as in the proof of D'Ambrosio \cite[Theorem 2.5]{D}, one has the following result.
\begin{lemma}\label{divlemf} Let $(M,g)$, $N$ and $\Omega$ be as in Assumption \ref{assufreecur1}.
Given $p>1$, suppose that $X\in L^1_{\lo}(T{\Omega_N})$ is a vector field  and  $f_X\in L^1_{\lo}({\Omega_N})$ is a non-negative function satisfying

\begin{itemize}
 \item[(i)] $f_X\leq \di X$ in the weak sense, i.e.,
 \[
 \int_{\Omega_N} u f_X {\dvol}_g\leq -\int_{\Omega_N} g( \nabla u, X ) {\dvol}_g, \ \forall \,u\in C^\infty_0({\Omega_N}) \text{ with }u\geq 0;
 \]

  \item[(ii)] $|X|^p/f_X^{p-1}\in L^1_{\lo}({\Omega_N})$, i.e., $\ds\int_K \left||X|^p/f_X^{p-1}\right|{\dvol}_g<\infty$   for any compact subdomain $K\subset {\Omega_N}$.
\end{itemize}

\noindent Then  we have
\begin{align*}
p^p \ds\ds\int_{\Omega_N}|\nabla u|^p\frac{|X|^p}{f_X^{p-1}} {\dvol}_g    \geq \ds\ds\int_{\Omega_N} |u|^p f_X{\dvol}_g,\ \forall\,u\in C^\infty_0({\Omega_N}).
\end{align*}
\end{lemma}

Now we consider the $p$-Laplacian operator on $C^\infty({\Omega_N})$.
Given $p>1$, the {\it   $p$-Laplacian}  is defined as
\[
\Delta_{p}(f):=\di \left(  |\nabla f|^{p-2} \nabla f  \right),\ \forall\,f\in C^\infty({\Omega_N}).
\]
Following D'Ambrosio and Dipierro \cite{DD},  we say that a function $f(x)\in W^{1,p}_{\lo}({\Omega_N})$ satisfies {\it $-\Delta_{p}(cf)\geq 0$ in the weak sense} for some $c\in \mathbb{R}$ if
\[
c\int_{{\Omega_N}} |\nabla f|^{p-2}g(\nabla f,\nabla u) {\dvol}_g\geq 0, \ \forall \,u\in C^\infty_0({\Omega_N}) \text{ with }u\geq 0.%\tag{3.1}\label{weak2}
\]

\begin{lemma}\label{mainlemmforcr}Let $(M,g)$, $N$ and $\Omega$ be as in Assumption \ref{assufreecur1} and let $\rho\in C^1({\Omega_N})$ be a non-negative nonconstant function.
Suppose that $p\in(1,\infty)$, $\alpha,\beta\in \mathbb{R}$ and $\rho$  satisfy the following conditions:
\begin{itemize}

\item[(i)] $\rho^{(\alpha-1)(p-1)}|\nabla\rho|^{p-1},\rho^\beta|\nabla \rho|^p,\rho^{p+\beta}\in L^1_{\lo}({\Omega_N})$;

\item[(ii)] $-\Delta_{p}(c\rho^\alpha)\geq 0$ in the weak sense, where $c=\alpha[(\alpha-1)(p-1)-\beta-1]\neq 0$.

\end{itemize}

\noindent Then we have
\begin{align*}
\int_{{{\Omega_N}}} |\nabla u|^p\rho^{p+\beta}{\dvol}_g\geq (\vartheta_{\alpha,\beta,p})^p\int_{{{\Omega_N}}} |u|^p\rho^\beta |\nabla\rho|^p {\dvol}_g,\ \forall\,u\in C^\infty_0({{\Omega_N}}),\tag{3.1}\label{5.1.1}
\end{align*}
 where $\vartheta_{\alpha,\beta,p}:={|(\alpha-1)(p-1)-\beta-1|}/{p}$.
\end{lemma}

\begin{proof}
Provided that $-\Delta_{p} (\rho^\alpha)\geq 0$ and $c>0$, we set
\[
X:=-\alpha \rho^{\beta+1}|\nabla\rho|^{p-2}\nabla \rho,\ \ \ f_X:=c |\nabla\rho|^p \rho^\beta.
\]
Clearly, $f_X\in L^1_{\lo}({{\Omega_N}})$. The H\"older inequality together with Condition (i) implies $\rho^{\beta+1}|\nabla\rho|^{p-1}\in L^1_{\lo}({{\Omega_N}})$ and hence,
$X\in L^1_{\lo}(T{{\Omega_N}})$.
Moreover,  $|X|^p/f^{p-1}_X\in L^1_{\lo}({{\Omega_N}})$ because $\rho^{p+\beta}\in L^1_{\lo}({{\Omega_N}})$.

 Given $\epsilon>0$ and $u\in C^\infty_0({{\Omega_N}})$ with $u\geq 0$, set $ v:=(\rho+\epsilon)^{-{c}/\alpha}u$ if $-c/\alpha-1<0$, otherwise set $v:=\rho^{-{c}/\alpha}u$.
 Condition (ii) then furnishes $\int_{{{\Omega_N}}}   |\nabla \rho^\alpha|^{p-2}g(\nabla \rho^\alpha, \nabla v) {\dvol}_g\geq 0$, which together with Lebesgue's  dominated convergence theorem  yields
\[
\int_{{{\Omega_N}}} c\rho^\beta|\nabla\rho|^pu{\dvol}_g\leq\int_{{{\Omega_N}}}\alpha \rho^{\beta+1}|\nabla\rho|^{p-2}g( \nabla \rho,\nabla u){\dvol}_g,
\]
that is, $f_X\leq \di X$ in the weak sense. Now (\ref{5.1.1}) follows from Lemma \ref{divlemf}.

If $-\Delta_{p} (\rho^\alpha)\leq 0$ and $c<0$,  set $X:=\alpha \rho^{\beta+1}|\nabla\rho|^{p-2}\nabla \rho$ and $f_X:=-c |\nabla\rho|^p \rho^\beta$. Then  the proof of this part follows in a similar manner.
\end{proof}

The same argument as the one in D'Ambrosio \cite[p.\,458]{D} furnishes the
following result.

\begin{lemma}\label{impsharplem}
Let $(M,g)$, $N$, $\Omega$, $\rho$ and $p,\alpha,\beta$ be as in  Lemma \ref{mainlemmforcr}.  Additionally suppose
\begin{itemize}

\item[(i)] there exists some $s>0$ such that ${\Omega_N}^-_s:=\rho^{-1}(-\infty,s]$ and ${{\Omega_N}}^+_s:=\rho^{-1}(s,\infty)$ are non-empty subsets of $M$ with piecewise regular boundaries;

\item[(ii)] there exists some $\epsilon_0>0$ such that
 \[
 0<\int_{{{\Omega_N}}^-_s} \rho^{c(\epsilon)p+\beta} |\nabla \rho|^p {\dvol}_g<+\infty,\ 0<\int_{{{\Omega_N}}^+_s} \rho^{-c(\epsilon)p+\beta} |\nabla \rho|^p {\dvol}_g<+\infty,\ \forall\,\epsilon\in (0,\epsilon_0),
 \]
where $c(\epsilon):= \frac{|(\alpha-1)(p-1)-\beta-1|+\epsilon}{p}$.
\end{itemize}
Then we have
\[
c(\epsilon)^p\int_{\Omega_N} |v_\epsilon|^p\rho^\beta |\nabla\rho|^p   {\dvol}_g>\int_{\Omega_N} |\nabla v_\epsilon|^p\rho^{\beta+p}{\dvol}_g, \ \forall\,\epsilon\in (0,\epsilon_0),\tag{3.2}\label{fakeinq4.4}
\]
where  \begin{align*}v_\epsilon(x):=\left\{
\begin{array}{lll}
\left(\frac{\rho(x)}{s}\right)^{c(\epsilon)},&& \text{ if  }x\in{\Omega_N}^-_s,\\
\tag{3.3}\label{defnitionvepsi}\\
\left(\frac{\rho(x)}{s}\right)^{-c(\epsilon/2)},&& \text{ if  }x\in {\Omega_N}^+_s.
\end{array}
\right.
\end{align*}
\end{lemma}

\subsection{Hardy inequalities with distance weights}
In this subsection, we study distance-weighted Hardy inequalities and prove Theorem \ref{keycompactHardy}. We begin by introducing the following convention.

\medskip

\noindent\textbf{Convention.} Let $(M,g)$, $N$ and $\Omega$ be as in Assumption \ref{assufreecur1}. We say that {\it $(M,g)$, $N$ and $\Omega$ satisfy ``Condition (\textbf{C})"} if one of the following conditions holds:
\begin{itemize}

\item  $\mathbf{Ric}_M\geq 0$, $n=0$, $N=\{o\}\subset \Omega$;

\smallskip

\item  $\mathbf{Ric}_M\geq 0$, $n=m-1$ and $N=\partial\Omega$ is weakly mean convex;

\smallskip

\item  $\mathbf{K}_M\geq 0$, $1\leq n\leq m-1$, $N$ is minimal and $\Omega\cap N\neq\emptyset$.
\end{itemize}

\medskip

\begin{definition}\label{Dsdef} Let $(M,g)$, $N$ and $\Omega$ be as in Assumption \ref{assufreecur1}.
Given $p>1$ and $ \beta\in \mathbb{R}$,   define  a norm of $u\in C^\infty(\Omega_N)$ as
\begin{align*}
\|u\|_{D}&:=\left( \int_{\Omega_N}|u|^p{r}^{\beta} {\dvol}_g+\int_{\Omega_N} |\nabla u|^p {r}^{p+\beta} {\dvol}_g\right)^{\frac1p}.
\end{align*}
The Banach space $D^{1,p}(\Omega_N,r^{\beta+p})$  is the completion of   $C^\infty_0({\Omega_N})$ with respect to the norm $\|\cdot\|_D$.
\end{definition}

Now we present the following result.
\begin{theorem}\label{firsthadyonpunctured}
Let $(M,g)$, $N$ and
$\Omega$ be as in Assumption \ref{assufreecur1} and satisfy Condition (\textbf{C}).
 Thus,  for any $1<p\neq (m-n)$ and $\beta< -(m-n)$, we have
 \begin{align*}
\int_{{\Omega}} |\nabla u|^pr^{p+\beta}{\dvol}_g\geq \left(\frac{\left|m-n+\beta\right|}{p}\right)^p\int_{{\Omega}} |u|^pr^\beta {\dvol}_g,\ \forall\,u\in C^\infty_0({\Omega_N}).\tag{3.4}\label{fristHadynonnegitavecurvature}
\end{align*}
 In particular,   (\ref{fristHadynonnegitavecurvature}) is optimal if $N$ is compact.
\end{theorem}
\begin{proof}We first consider the case when $n\geq 1$. Set
\[
\rho:=r|_{\Omega_N},\ \alpha:= [{p-(m-n)}]/({p-1}),\ c:=\alpha[(\alpha-1)(p-1)-\beta-1].\tag{3.5}\label{notation}
 \] Thus,
Theorem \ref{huperseacedata} and Corollary \ref{inlalpace} yield
\begin{align*}
-\Delta_p(c\rho^\alpha)=-|c|^{p-2}|\alpha|^p\left[ (\alpha-1)(p-1)-\beta-1 \right]r^{(\alpha-1)(p-1)-1}\left[ r\Delta r+(\alpha-1)(p-1) \right]\geq 0.
\end{align*}
Obviously, $\rho^{(\alpha-1)(p-1)}, \rho^\beta, \rho^{p+\beta}\in L^1_{\lo}(\Omega_N)$.
Hence, (\ref{fristHadynonnegitavecurvature}) follows from Lemma \ref{mainlemmforcr} immediately.

Now we show that (\ref{fristHadynonnegitavecurvature}) is sharp if  $N$ is compact, in which case $\min_{\mathbf{n}\in \mathcal {V}SN}c_\mathcal {V}(\mathbf{n})>0$. Without loss of generality, we may assume $N\subset \overline{\Omega}$ (otherwise consider $\widetilde{N}:=N\cap \overline{\Omega}$). Choose a small $s
\in (0, \min_{\mathbf{n}\in \mathcal {V}SN}c_\mathcal {V}(\mathbf{n}))$ such that both $T_s\cap \Omega$ and $\Omega\setminus T_s$ are non-empty, where $T_s$ is the $s$-tuber neighborhood of $N$. Clearly,
 $N\subset T_s\cap \overline{\Omega}\subset \overline{\Omega}$.
For any $\epsilon\in (0,1/2)$, we define    $v_\epsilon$ as in (\ref{defnitionvepsi}), where $\rho$ and $\alpha$ are given by (\ref{notation}).
It is not hard to check that
 \[
 c(\epsilon)p+\beta+(m-n-1)\geq -1+\epsilon,\ -c(\epsilon/2)p+\beta+(m-n-1)\leq -1- {\epsilon}/{2},\tag{3.6}\label{indexforfinite}
 \]
 where $c(\epsilon)=(|m-n+\beta|+\epsilon)/p$.
Let $(t,\mathbf{n})$ denote Fermi coordinates. Thus, Theorems \ref{importantlemmahk} and  \ref{hypernonlaplac} yield $\det \mathcal {A}(t,\mathbf{n})\leq t^{m-n-1}$, which together with (\ref{2.0'fermico}) and (\ref{indexforfinite}) furnishes
\begin{align*}
&\int_{\Omega_N} |v_\epsilon|^p\rho^\beta |\nabla\rho|^p   {\dvol}_g=\int_{\Omega_N} |v_\epsilon|^pr^\beta {\dvol}_g
=\frac{1}{s^{c(\epsilon)p}}\int_{T_s \cap \Omega} r^{c(\epsilon)p+\beta}{\dvol}_g+\frac{1}{s^{-c(\epsilon/2)p}}\int_{\Omega\backslash T_s}r^{-c(\epsilon/2)p+\beta}{\dvol}_g\\
\leq & c_{m-n-1}\vol_{i^*g}(N)\left[\frac{1}{s^{c(\epsilon)p}}\int_0^s t^{c(\epsilon)p+\beta+(m-n-1)}dt+\frac{1}{s^{-c(\epsilon/2)p}}\int_s^\infty t^{-c(\epsilon/2)p+\beta+(m-n-1)}dt \right]<+\infty.\tag{3.7}\label{veprfinite}
\end{align*}
By using (\ref{veprfinite}), one can easily show that  Lemma \ref{impsharplem} holds (for $\rho=r|_{\Omega_N},\ \alpha= [{p-(m-n)}]/({p-1}) $) and hence,
  $\|v_\epsilon\|_D<+\infty$.

On the other hand, set $v_{\epsilon,\iota}:=\max\{v_\epsilon-\iota,0 \}$ for small $\iota>0$. Because $v_{\epsilon,\iota}$ is a globally Lipschitz function with compact
support in ${\Omega_N}$, a standard argument (cf.  Hebey \cite[Lemma 2.5]{H} or Meng et al. \cite[Lemma A.1]{MWZ}) yields  $v_{\epsilon,\iota}\in D^{1,p}({\Omega_N},r^{\beta+p})$.
Since $v_\epsilon\geq 0$, $\|v_\epsilon\|_D<+\infty$ and  $\chi_{\{  v_\epsilon\geq \iota\}} |\iota|^pr^{\beta}\leq |v_\epsilon|^pr^{\beta}\in L^1(\Omega_N)$, the dominate convergence theorem yields
\begin{align*}
&\|v_{\epsilon,\iota}-v_\epsilon\|_D^p=\int_{{\Omega_N}}|v_{\epsilon,\iota}-v_\epsilon|^p{r}^{\beta} \dvol_g + \int_{{\Omega_N}} |\nabla (v_{\epsilon,\iota}-v_\epsilon)|^p r^{p+\beta} \dvol_g\\
=&\int_{\Omega_N} \chi_{\{0\leq v_\epsilon\leq \iota\}} |v_\epsilon|^p r^\beta \dvol_g+\int_{\Omega_N} \chi_{\{  v_\epsilon\geq \iota\}} |\iota|^p r^{\beta}\dvol_g+\int_{\Omega_N} \chi_{\{0\leq v_\epsilon\leq \iota\}} |\nabla v_\epsilon|^p r^{p+\beta}\dvol_g\rightarrow 0,\tag{3.8}\label{postivefunctionconvergence}
\end{align*}
as $\iota\rightarrow 0^+$.
Hence, $v_\epsilon\in D^{1,p}({\Omega_N},r^{\beta+p})$, which furnishes
a sequence $u_j\in C^\infty_0({\Omega_N})$ such that
\[
\int_{{\Omega}} |u_j|^p r^\beta {\dvol}_g\rightarrow \int_{{\Omega}} |v_\epsilon|^p r^\beta  {\dvol}_g,\ \ \ \int_{{\Omega}} |\nabla u_j|^pr^{p+\beta}{\dvol}_g\rightarrow \int_{{\Omega}} |\nabla v_\epsilon|^pr^{p+\beta}{\dvol}_g, \text{ as }j\rightarrow \infty.
\]
On account  of (\ref{fristHadynonnegitavecurvature}), (\ref{fakeinq4.4}) and (\ref{notation}),  we get
\begin{align*}
\left(\frac{\left|m-n+\beta\right|}{p}\right)^p\leq \inf_{u\in C^\infty_0({{\Omega_N}})\backslash\{0\}}\frac{\int_{{{\Omega}}}|\nabla u|^p  r^{p+\beta}{\dvol}_g}{\int_{{{\Omega}}}|u|^p r^\beta  {\dvol}_g} \leq \lim_{j\rightarrow\infty}\frac{\int_{{{\Omega}}} |\nabla u_j|^pr^{p+\beta}{\dvol}_g}{\int_{{{\Omega}}}|u_j|^p r^\beta {\dvol}_g}< c(\epsilon)^p\rightarrow \left(\frac{\left|m-n+\beta\right|}{p}\right)^p,
\end{align*}
as $\epsilon\rightarrow 0^+$.
So the sharpness of (\ref{fristHadynonnegitavecurvature}) follows.

In view of Remark \ref{zerodimensinocase}, the proof of the case when $n=0$ follows in a similar manner, which is omitted.
\end{proof}

In some cases it is impossible to  extend (\ref{fristHadynonnegitavecurvature})  to   $C^\infty_0(\Omega)$. For example, $\Omega=M$ is compact.
However, we will see that (\ref{fristHadynonnegitavecurvature}) remains valid   for the following space
\[
C_0^\infty(\Omega,N):=\{u\in C_0^\infty(\Omega): \,u(N)=0 \}.%\tag{3.11}\label{specialclosedspce}
\]

\begin{definition}\label{Hanshukongj} Let $(M,g)$, $N$ and
$\Omega$ be as in Assumption \ref{assufreecur1} and let $U\subset M$ be an open set.
Given $p>1$ and $ \beta\in \mathbb{R}$ with $p+\beta>-(m-n)$,   define  a norm of $u\in C^\infty(U)$ as
\begin{align*}
\|u\|_{p,\beta}&:=\left( \int_{U}|u|^p{r}^{p+\beta} {\dvol}_g+\int_{U} |\nabla u|^p {r}^{p+\beta} {\dvol}_g\right)^{\frac1p}.\tag{3.9}\label{weightednorm}
\end{align*}
The {\it weighted Sobolev space} ${W^{1,p}_0}(U, {r}^{p+\beta})$ (resp., ${W^{1,p}}(U, {r}^{p+\beta})$) is the completion of  $C^{\infty}_{0}(U)$ (resp., $C^{\infty}_{p,\beta}(U):=\{ u\in C^\infty(U):\, \|u\|_{p,\beta}<\infty\}$) under the norm $\|\cdot\|_{p,\beta}$. In particulary, if $p+\beta=0$, the $W^{1,p}_0(U,r^{p+\beta})$  is the standard Sobolev space $W^{1,p}_0(U)$.
\end{definition}
It follows from Lemma \ref{centerzeroinfite}  that (\ref{weightednorm}) is well-defined.
In Sections \ref{basistheorem}-\ref{non-compacthary} we mainly consider $W^{1,p}_0(\Omega_N,r^{p+\beta})$ while in Appendix \ref{Soblevspace} we investigate the relation between $W^{1,p}(M,r^{p+\beta})$ and $W^{1,p}_0(\Omega_N,r^{p+\beta})$.
The properties of weighted Sobolev spaces are studied in Appendix \ref{Soblevspace}, which furnish the following result.
\begin{lemma}\label{continuouslemma}
Let $(M,g),N,\Omega$ and $m,n,p,\beta$ be as in Definition \ref{Hanshukongj}.
Thus
\[
C^\infty_0(\Omega,N)\subset W^{1,p}_0(\Omega_N,r^{p+\beta}), \text{ that is}, \ u\in C^\infty_0(\Omega,N)\Longrightarrow u|_{\Omega_N}\in  W^{1,p}_0(\Omega_N,r^{p+\beta}).\tag{3.10}\label{functionsspacessubset}
\]
\end{lemma}
\begin{proof}
Since the proof is trivial if $N= \partial\Omega$, it is sufficient to show the lemma when
$N\cap \Omega\neq\emptyset$. Given $u\in C^\infty_0(\Omega,N)$,  the zero extension furnishes $u\in W^{1,p}(M,r^{p+\beta})$.  Since $u$ is continuous in $M$ and $u=0$ in $M\backslash\Omega_N$, Theorem \ref{maintheoreminaapendix} yields $u|_{\Omega_N}\in   W^{1,p}_0(\Omega_N,r^{p+\beta})$.
\end{proof}

Lemma \ref{continuouslemma} together with Theorem \ref{firsthadyonpunctured} yields the following result.
\begin{proposition}\label{reverRicinfty}
Let $(M,g)$, $N$ and
$\Omega$ be as in Assumption \ref{assufreecur1} and satisfy Condition (\textbf{C}).
Given $1<p\neq (m-n)$ and $\beta<-(m-n)$ with $p+\beta>-(m-n)$,
we have
 \[
\int_{\Omega}|\nabla u|^p r^{\beta+p} {\dvol}_g\geq \left( \frac{|m-n+\beta|}{p} \right)^p\int_{\Omega} {|u|^p}r^{\beta}{\dvol}_g, \ \forall\, u\in C_0^\infty(\Omega,N).\tag{3.11}\label{new3.26re}
\]
In particular, if $N$ is compact, then $\left(  {|m-n+\beta|}/{p} \right)^p$ is best with respect to $C_0^\infty(\Omega,N)$.
\end{proposition}
\begin{proof}
Given $u\in W^{1,p}_0(\Omega_N, r^{p+\beta})$, there is a sequence $u_i\in C^\infty_0(\Omega_N)$ such that $u_i$ converges to $u$ under $\|\cdot\|_{p,\beta}$.  By passing a subsequence and using Lemma  \ref{nullmeaure}, we can assume that $u_i$ converges to $u$ pointwise  a.e. (w.r.t. $\vol_g$) in $M$.   Thus,  Fatou's lemma together with (\ref{fristHadynonnegitavecurvature}) yields
\begin{align*}
&\int_{\Omega}|\nabla u|^p r^{\beta+p} {\dvol}_g=\underset{i\rightarrow \infty}{\lim\inf}\int_{\Omega}|\nabla u_i|^p r^{\beta+p} {\dvol}_g\geq \left( \frac{|m-n+\beta|}{p} \right)^p \underset{i\rightarrow \infty}{\lim\inf}\int_{\Omega} {|u_i|^p}r^{\beta}{\dvol}_g\\
\geq & \left( \frac{|m-n+\beta|}{p} \right)^p\int_{\Omega} \underset{i\rightarrow \infty}{\lim\inf}{|u_i|^p}r^{\beta}{\dvol}_g=\left( \frac{|m-n+\beta|}{p} \right)^p\int_{\Omega} {|u|^p}r^{\beta}{\dvol}_g,
\end{align*}
which combined with Lemma \ref{continuouslemma} yields
 (\ref{new3.26re}).
It remains to show that  (\ref{new3.26re}) is sharp when $N$ is compact. Since $C^\infty_0(\Omega_N)\subset C^\infty_0(\Omega,N)$, the sharpness follows from (\ref{new3.26re}) and  Theorem \ref{firsthadyonpunctured} directly.
\end{proof}

In the following, we study the improved Hardy inequality.
 For any  $p>1$ and $\beta<-(m-n)$, set
 \begin{align*}
\mathscr{I}[u]:=\int_{{\Omega}} |\nabla u|^pr^{p+\beta} {\dvol}_g-\left(\frac{|m-n+\beta|}{p}\right)^p\int_{{\Omega}} {|u|^p}{r^\beta}  {\dvol}_g,\ \forall\,u\in C^\infty_0({\Omega_N}).\tag{3.12}\label{FunctionalI}
\end{align*}

Following Barbatis et al. \cite{BFT}, we have the following estimate.
\begin{lemma}\label{refinedsequence}
Let $(M,g)$, $N$ and
$\Omega$ be as in Assumption \ref{assufreecur1} and satisfy Condition (\textbf{C}). Given $p>1$ and $\beta<-(m-n)$, there exists a constant $C=C(p)>0$ such that  the following estimates hold:

\begin{itemize}
\item[(a)] if $p\in (1,2)$, then
$\ds\mathscr{I}[u]\geq C\int_\Omega \frac{|\nabla v|^2 r^{2-(m-n)}}{\left( |\delta v| +|r\nabla v| \right)^{2-p}}{\dvol}_g,\ \forall\,u\in C^\infty_0(\Omega_N)$;

\item[(b)] if $p\in [2,\infty)$, then
$\ds\mathscr{I}[u]\geq C\int_\Omega |\nabla v|^p r^{p-(m-n)}{\dvol}_g,\ \forall\,u\in C^\infty_0(\Omega_N).$
\end{itemize}

\smallskip

\noindent Here, $\delta:=(m-n+\beta)/p$ and $v:=u\, r^\delta$.
\end{lemma}
\begin{proof}
Set $X:= \delta v \nabla r$ and $Y:=r\nabla v$. The divergence theorem together with Theorem \ref{huperseacedata}, Corollary \ref{inlalpace} and Remark \ref{zerodimensinocase} yields
\begin{align*}
\int_\Omega \frac{|X|^{p-2}}{r^{m-n}}\, g( X,Y) {\dvol}_g
=\frac{-|\delta|^{p-2}\delta}{p}\int_\Omega \frac{|v|^p}{r^{m-n}}\left[ r\Delta r+1-(m-n)\right]{\dvol}_g\leq 0.\tag{3.13}\label{basisdivergence}
\end{align*}
On the other hand, it is straightforward to show
\begin{align*}
\mathscr{I}[u]=\int_\Omega \frac{|X-Y|^p-|X|^p}{r^{m-n}}{\dvol}_g.\tag{3.14}\label{middelI}
\end{align*}

Providing $p\in (1,2)$,  a standard inequality (cf.  Barbatis et al. \cite[Lemma 3.1]{BFT}) yields
\[
|X-Y|^p-|X|^p\geq C \frac{|Y|^2}{\left( |X|+|Y| \right)^{2-p}}-p|X|^{p-2}\, g(X,Y),\ \  \forall\,X,Y\in TM,
\]
which combined with (\ref{middelI}) and (\ref{basisdivergence}) furnishes  Statement (a). Alternatively, provided $p\geq 2$, one has
\[
|X-Y|^p-|X|^p\geq C|Y|^p-p|X|^{p-2}\,g(X,Y),\ \  \forall\,X,Y\in TM.
\]
Now Statement (b) can be proved in the same  way as shown
above.
\end{proof}

\begin{proposition}\label{notachivedconstant}
Let $(M,g)$, $N$ and
$\Omega$ be as in Assumption \ref{assufreecur1} and satisfy Condition (\textbf{C}). Additionally suppose that $N$ is compact. Given $1<p\neq(m-n)$ and $\beta<-(m-n)$, the following statements are true:

\begin{itemize}

\item[(i)] there holds
\begin{align*}
\inf_{u\in C^\infty_0(\Omega_N)\backslash\{0\}}\frac{\int_{{\Omega}} |\nabla u|^pr^{p+\beta} {\dvol}_g}{\int_{{\Omega}}  {|u|^p}{r^\beta} {\dvol}_g}=\left(\frac{\left|m-n+\beta\right|}{p}\right)^p,
\end{align*}
but the constant cannot be achieved;

\item[(ii)] if $p+\beta>-(m-n)$, then
\begin{align*}
\inf_{u\in C^\infty_0(\Omega,N)\backslash\{0\}}\frac{\int_{{\Omega}} |\nabla u|^p r^{p+\beta}{\dvol}_g}{\int_{{\Omega}}  {|u|^p}{r^\beta} {\dvol}_g}= \left(\frac{\left|m-n+\beta\right|}{p}\right)^p, \tag{3.15}\label{sharpclsoedconstant}
\end{align*}
but the constant cannot be achieved.
\end{itemize}

\end{proposition}

\begin{proof}
Thanks to Theorem   \ref{firsthadyonpunctured} and Proposition \ref{reverRicinfty}, it suffices to show that the constants cannot be achieved.

(i) Suppose  otherwise, i.e., the constant could be achieved by some $u\in C^\infty_0(\Omega_N)$.
Thus, Lemma \ref{refinedsequence} implies $0=\nabla v=\nabla (u\, r^\delta)$ a.e. and hence, there exists a constant $C$ such that $r^{-\delta}=C u\in D^{1,p}(\Omega_N,r^{p+\beta})$ (see Definition \ref{Dsdef}).
However, this leads to a contradiction by Lemma \ref{centerzeroinfite}/(ii).

(ii)  By an argument similar to the one used in the proof of Proposition \ref{reverRicinfty}, one can show that
 Lemma \ref{refinedsequence} is valid for $u\in D^{1,p}(\Omega_N,r^{p+\beta})$.
Now we claim $W^{1,p}_0(\Omega_N,r^{p+\beta})\subset D^{1,p}(\Omega_N,r^{p+\beta})$. If the claim is true,
then one can conclude the proof by (\ref{functionsspacessubset}) and the same argument as above.

To prove the claim, given $u\in W^{1,p}_0(\Omega_N,r^{p+\beta})$,  choose a sequence $u_i\in C^\infty_0(\Omega_N)$ converging to $u$ under $\|\cdot\|_{p,\beta}$. By passing a subsequence and using Lemma \ref{nullmeaure}, we may assume that $u_i$ converges to $u$ pointwise a.e. in $\Omega$.
On the other hand, (\ref{new3.26re})  implies that $(u_i)$ is also a Cauchy sequence in $D^{1,p}(\Omega_N,r^{p+\beta})$, whose limit is denoted by $\tilde{u}\in D^{1,p}(\Omega_N,r^{p+\beta})$.
Obviously, $\int_{\Omega}|\nabla u-\nabla \tilde{u}|^pr^{p+\beta}\dvol_g=0$. Moreover, since (\ref{new3.26re}) holds for $ D^{1,p}(\Omega_N,r^{p+\beta})$,
 Fatou's lemma together with $\tilde{u}- u_i\in D^{1,p}(\Omega_N,r^{p+\beta})$ yields
\begin{align*}
\int_{{\Omega}}  {|\tilde{u}-u|^p}{r^\beta} {\dvol}_g=&\int_{{\Omega}} \underset{i\rightarrow \infty}{\lim\inf} {|\tilde{u}-u_i|^p}{r^\beta} {\dvol}_g\leq \underset{i\rightarrow \infty}{\lim\inf}\int_{{\Omega}} {|\tilde{u}-u_i|^p}{r^\beta} {\dvol}_g \\
\leq& \left(\frac{|m-n+\beta|}{p}\right)^{-p}\underset{i\rightarrow \infty}{\lim\inf}\int_{{\Omega}} |\nabla (\tilde{u}- u_i)|^p r^{p+\beta} {\dvol}_g=0,
\end{align*}
which implies $u=\tilde{u}\in D^{1,p}(\Omega_N,r^{p+\beta})$. So the claim is true.
\end{proof}

\begin{proof}[Proof of Theorem \ref{keycompactHardy}]Owing to $k=m-n$,
Theorem \ref{keycompactHardy} follows from Theorem \ref{firsthadyonpunctured}, Proposition \ref{reverRicinfty} and Proposition \ref{notachivedconstant} directly.
\end{proof}

\subsection{Hardy inequalities with logarithmic weights}
In the sequel, we study the  logarithm-weighted Hardy inequalities. First, we need the following result,
whose proof is trivial.

\begin{lemma}\label{interglemaGa}Given $0<L<D$,
for any $s_1,s_2\in \mathbb{R}$ and $l\in (L,D]$, set
\[
H_1(s_1,s_2):=\int^L_0 \left[ \log\left( \frac{D}{t} \right)  \right]^{s_1} t^{s_2} dt,\ H_2(l,s_1,s_2):=\int^l_L \left[ \log\left( \frac{D}{t} \right)  \right]^{s_1} t^{s_2} dt.
\]
Thus, we have
\begin{itemize}
\item $H_1$ is well-defined if either $s_1\in \mathbb{R},\,s_2>-1$ or $s_1<-1,\, s_2=-1$.

\item $H_2$ is well-defined if either $s_1,s_2\in \mathbb{R},\,l<D$ or $s_1>-1,\,s_2\in \mathbb{R},\,l=D$.
\end{itemize}
\end{lemma}

\medskip

 %We say that a domain $\Omega$ in $M$ is   {\it starlike with respect to $N$} if for every $x\in \Omega$, there exists a minimal geodesic $\gamma(t)$, $0\leq t\leq 1$ from $N$ to $x$ such that $\gamma((0,1])\subset \Omega$ and $\dot{\gamma}(0)\in \mathcal {V}N$.

 \begin{theorem}\label{logriccnonclosed}
Let $(M,g)$, $N$ and
$\Omega$ be as in Assumption \ref{assufreecur1} and satisfy Condition (\textbf{C}).
If the constants  $D,p,\beta\in \mathbb{R}$ and $\alpha\in \mathbb{R}\backslash\{0\}$ satisfy
\[
 \sup_{x\in \Omega}r(x)\leq D, \ p\geq m-n>1, \  \log\left( \frac{D}{\sup_{x\in \Omega}r(x)} \right)(m-n-p)\leq (\alpha-1)(p-1)< \beta+1,%\tag{2.10}\label{logconditnum}
\]
then we have
\begin{align*}
\int_{\Omega} \left[\log\left(\frac{D}{r} \right)\right]^{p+\beta}|\nabla u|^p  {{\dvol}_g}\geq (\vartheta_{\alpha,\beta,p})^p\int_{\Omega} \left[\log\left(\frac{D}{r} \right)\right]^{\beta} \frac{|u|^p}{r^p} {{\dvol}_g},\ \forall\,u\in C^\infty_0({\Omega_N}),\tag{3.16}\label{loginnonclosed}
\end{align*}
 where $\vartheta_{\alpha,\beta,p}:= {[\beta+1-(\alpha-1)(p-1)]}/{p}$.
 In particular, $(\vartheta_{\alpha,\beta,p})^p$ is sharp if $N$ is compact and
 \begin{align*}
 p=m-n,\ \alpha=1, \ \Omega={T}_D,\ D=\sup_{x\in \Omega}r(x).\tag{3.17}\label{constequali}
 \end{align*}
 \end{theorem}

 \begin{remark}
 Since $m-n>1$, the case when $N=\partial \Omega$ is eliminated.
 \end{remark}

 \begin{proof}
Set $\rho:=\log\left( \frac{D}{r} \right)$ and $c:=\alpha\left[(\alpha-1)(p-1)-\beta-1  \right]$.
Then Theorem \ref{huperseacedata} and Remark \ref{zerodimensinocase} yield
\begin{align*}
-\Delta_{p}(c\rho^\alpha)
=|c|^{p-2}|\alpha|^p        \frac{\rho^{(\alpha-1)(p-1)-1}}{r^p}       \left[(\alpha-1)(p-1)-\beta-1  \right]  \left[- (\alpha-1)(p-1) +\rho(-p+1+r\Delta  r) \right]\geq0.
\end{align*}
Since $\rho$ is continuous on any compact subset of $\Omega_N$, we have $\rho^{(\alpha-1)(p-1)}|\nabla\rho|^{p-1},\rho^{p+\beta}, \rho^\beta|\nabla \rho|^p\in L^1_{\lo}({{\Omega_N}})$.
 Then (\ref{loginnonclosed}) follows from Lemma \ref{mainlemmforcr} immediately.

In the sequel, we show that (\ref{loginnonclosed}) is optimal if $N$ is compact and (\ref{constequali}) holds. Set $L:=D/2>0$ and  $s:=\log\left(\frac{D}{L}\right)=\log2>0$.
Since $\Omega$ is a natural domain, it is easy to see that
\[
\Omega^-_{N_s}:=\rho^{-1}(-\infty,s]=\Omega_N\backslash T_{ {{L}}},\ \ {\Omega}^+_{N_s}:=\rho^{-1}(s,\infty)=\Omega_N \cap T_{L},
\]
 are non-empty subsets  with piecewise smooth boundaries. Moreover, we have
 \begin{align*}
 c(\epsilon)p+\beta=2\beta+1+\epsilon>-1,\
 -c(\epsilon/2)p+\beta= -1-\epsilon/2<-1,\ m-n-p-1=-1,
 \end{align*}
 where $c(\epsilon)=(\beta+1+\epsilon)/{p}$.
By using Theorem \ref{importantlemmahk},  Remark \ref{zerodimensinocase} and Lemma \ref{interglemaGa}, one has
\begin{align*}
0<\int_{\Omega^-_{N_s}} \rho^{c(\epsilon)p+\beta} |\nabla \rho|^p {\dvol}_g&\leq c_{m-n-1}\vol_{i^*g}(N)\int^{D}_{{L}}\log\left( \frac{D}{t} \right)^{c(\epsilon)p+\beta}t^{m-n-p-1}dt<+\infty,\\
0<\int_{{\Omega}^+_{N_s}} \rho^{-c(\epsilon)p+\beta} |\nabla \rho|^p {\dvol}_g&\leq c_{m-n-1}\vol_{i^*g}(N)\int_0^{{L}}\log\left( \frac{D}{t} \right)^{-c(\epsilon)p+\beta}t^{m-n-p-1}dt<+\infty.
\end{align*}
Given $\epsilon\in (0,1)$, let $v_\epsilon$ be defined as in (\ref{defnitionvepsi}). Thus, we obtain (\ref{fakeinq4.4}) by Lemma \ref{impsharplem}.
For any $\iota\in (0,1)$,  $v_{\epsilon,\iota}:=\max\{v_\epsilon-\iota,0\}$ is a globally Lipschitz function with compact support in ${\Omega_N}$.  The remainder of the proof is   analogous to that of Theorem \ref{firsthadyonpunctured} and hence, is omitted here.
 \end{proof}

By a suitable modification to the proof of Proposition \ref{reverRicinfty}, one can extend  (\ref{loginnonclosed}) to $u\in C^\infty_0(\Omega,N)$.
 We leave it to the interested readers.  Refer to  Meng et al. \cite[Theorem 3.3]{MWZ} for a weighted-Ricci-version in the case of  $n=0$.

\section{Sharp Hardy inequalities via general submanifolds}\label{non-compacthary}
In the previous section, we establish two kinds of weighted Hardy inequalities which are sharp when   submanifolds are compact. In this section, we will continue to investigate sharp Hardy inequalities of distance weights in the case when  submanifolds are possibly non-compact.

\subsection{Improved Hardy inequalities}
According to Lemma \ref{refinedsequence}, one can improve (\ref{fristHadynonnegitavecurvature}) by adding some nonnegative correction terms in the right-hand side.
Inspired by Barbatis et al. \cite{BFT}, we study the logarithmic correction.

\begin{definition}\label{generalizeddivergence}Let $(M,g)$ be a complete Riemannian manifold and let $i:N\hookrightarrow M$ be a closed submanifold. Suppose that ${\dmu}$ is a measure on $M$ and is smooth in $M\backslash N$.
The {\it divergence of a $C^1$-vector field $X$ with respect to ${\dmu}$ in $M\backslash N$} is defined by $\di_\mu X {\dmu}:=d(X\rfloor {\dmu})$ (cf.\,(\ref{innerexterproduct})).
\end{definition}

The definition above is natural. In fact,
 $\di_{\vol_g}=\di $, where $\di$ denotes the standard divergence. In the sequel, we always choose ${\dmu}:=r^{p+\beta}\dvol_g$. A direct calculation then yields
\[
\di_\mu X=\di X+(p+\beta)\frac{g(X,\nabla r)}{r}\   \text{ in }M\backslash N.\tag{4.1}\label{divergncedef}
\]
It is not hard to check  $\di_\mu(fX)=f\di_\mu X+g(X, \nabla f)$ for any $f\in C^\infty(M\backslash N)$.

\medskip

For convenience, we introduce the following assumption in this section.
\begin{assumption}\label{assufreecur}
Let $(M,g)$ be an $m$-dimensional complete  Riemannian manifold, let $i:N\hookrightarrow M$ be an $n$-dimensional closed submanifold and let $\Omega$ be a natural domain in $M$ with $\sup_{x\in \Omega}r(x)<+\infty$. In short,  $(M,g)$, $N$ and $\Omega$ satisfy  Assumption \ref{assufreecur1} and  $\sup_{x\in \Omega}r(x)<+\infty$.
 \end{assumption}

\begin{theorem}\label{mainnon-compacttheroem}
Let  $(M,g)$, $N$ and
$\Omega$ be as in Assumption \ref{assufreecur} and satisfy Condition (\textbf{C}).
 Thus,  for any $p>1$ and $\beta<-(m-n)$, there exists a constant $\mathcal {T}=\mathcal {T}(p,\beta,m-n)> 1$ such that for any $D\geq \mathcal {T}\sup_{x\in \Omega}r(x)$,
 \begin{align*}
\int_{{\Omega}} |\nabla u|^p r^{p+\beta} {\dvol}_g\geq |\delta|^p\int_{{\Omega}}  {|u|^p}{r^\beta} {\dvol}_g+\frac{p-1}{2p}|\delta|^{p-2}\int_\Omega  {|u|^p}{r^\beta}\log^{-2}\left( \frac{D}{r} \right){\dvol}_g ,\ \forall\,u\in C^\infty_0({\Omega_N}),
\end{align*}
where $\delta:=(m-n+\beta)/p$.
\end{theorem}
\begin{proof}
Let $X$ be a vector field on $\Omega$ and let ${\dmu}:=r^{p+\beta}\dvol_g$.
For any $u\in C^\infty_0(\Omega_N)$, we have
\begin{align*}
\di_\mu \left( |u|^p X \right)=|u|^p\di_\mu X+p|u|^{p-2} u\, g( \nabla u,X),
\end{align*}
which together with Stokes' theorem, H\"older's inequality  and  Young's inequality yields
\begin{align*}
\int_\Omega |u|^p\di_{\mu} X {\dmu}=&-p\int_\Omega |u|^{p-2} u\, g(\nabla u,X)  {\dmu}
\leq  p\left( \int_\Omega |\nabla u|^p{\dmu} \right)^{\frac1p}\left( \int_\Omega |X|^{\frac{p}{p-1}}|u|^p{\dmu} \right)^{\frac{p-1}{p}}\\
\leq & \int_\Omega|\nabla u|^p{\dmu}+(p-1)\int_\Omega |X|^{\frac{p}{p-1}}|u|^p{\dmu}.
\end{align*}
That is,
\[
\int_\Omega |\nabla u|^p {\dmu} \geq \int_\Omega \left[\di_{\mu} X-(p-1)|X|^{\frac{p}{p-1}}  \right]|u|^p{\dmu}.\tag{4.2}\label{middleineqaulity}
\]

Let $\mathcal {T}>1$ be a constant chosen later. Given any $D\geq \mathcal {T}\sup_{x\in \Omega} r(x)$,
set $\Psi(t):=[\log(D/t)]^{-1}$  and
\[
X(x):=\delta|\delta|^{p-2}\frac{\nabla r(x)}{r^{p-1}(x)}\left[ 1+\frac{p-1}{p\delta}\Psi\left( r(x) \right) +a\Psi^2\left(r(x) \right)  \right],
\]
where   $a$ is another constant chosen later such that
\[
1+\frac{p-1}{p\delta}\Psi\left(r(x)\right) +a\Psi^2\left( r(x) \right)>0,\ \forall\,x\in \Omega.\tag{4.3}\label{positivecosnta}
\]
Since  $u\in C^\infty_0(\Omega_N)$, the inequality (\ref{middleineqaulity}) is well-defined for this $X$. Furthermore, thanks to (\ref{middleineqaulity}), the theorem will be proved if we can show
\[
\di_{\mu} X-(p-1)|X|^{\frac{p}{p-1}}\geq \frac{|\delta|^p}{r^p}\left( 1+\frac{p-1}{2p\delta^2}\Psi^2(r) \right) \text{ in }\Omega_N.\tag{4.4}\label{refineclaime}
\]

In order to prove (\ref{refineclaime}), notice that
\[
\delta[r\Delta r+1-(m-n)]\geq 0,\ \nabla \Psi^{\alpha-1}(r)=(\alpha-1)\Psi^\alpha(r)r^{-1}\nabla r, \ \forall\,\alpha\neq1.\tag{4.5}\label{logdervative}
\]
Then a  cumbersome but direct calculation together with (\ref{divergncedef}) and (\ref{logdervative}) yields
\begin{align*}
\di_{\mu} X-(p-1)|X|^{\frac{p}{p-1}}
\geq\frac{|\delta|^{p}}{r^p}f\left( \Psi(r) \right),\tag{4.6}\label{midelcaim}
\end{align*}
where
\[
f(t):=p\left( 1+\frac{p-1}{p\delta}t+at^2  \right)+\frac{t^2}{\delta}\left( \frac{p-1}{p\delta}+2at \right)-(p-1)\left( 1+\frac{p-1}{p\delta}t+at^2 \right)^{\frac{p}{p-1}}.
\]
Given $t\geq0$, Taylor's formula furnishes
 \[
f(t)=f(0)+f'(0)t+\frac12 f''(\xi_t)t^2=1+\frac12 f''(\xi_t)t^2,\ 0\leq \xi_t\leq t.\tag{4.7}\label{Taylor}
\]
Note that $f'''(0)=\frac{6a}{\delta}-\frac{(2-p)(p-1)}{p^2\delta^3}>0$ if $a$ satisfies the following conditions:

\smallskip

\ \ \ \ (i) $a\in  \left(0, \frac{(2-p)(p-1)}{6p^2\delta^2} \right)$ if $1<p<2$;\ \ \ \
(ii) $a<\frac{(2-p)(p-1)}{6p^2\delta^2}\leq0$ if  $p\geq 2$.

\noindent Hence, we can find $\mathfrak{T}=\mathfrak{T}(p,\beta,m-n)>0$ and $a=a(p,\beta,m-n)$ such that $f'''(t)>0$ and $1+\frac{p-1}{p\delta}t+at^2>0$ for $t\in [0,\mathfrak{T}]$, which together with (\ref{Taylor}) implies
\[
f''(\xi_t)\geq f''(0)=\frac{p-1}{p\delta^2},\ \ \ \ f(t)\geq 1+\frac{p-1}{2p\delta^2}t^2,\ t\in [0,\mathfrak{T}].\tag{4.8}\label{festimates}
\]
On account of (\ref{midelcaim}) and (\ref{festimates}), one gets (\ref{refineclaime}) by choosing
$\mathcal {T}:=e^{\frac{1}{\mathfrak{T}}}$.
\end{proof}

\subsection{Sharpness}
In this subsection, we show that the Hardy inequality in Theorem \ref{mainnon-compacttheroem} is actually optimal. It is remarkable that all the estimates given here are {\it free of curvature}.

 Let $M,N$ and $\Omega$ be as in Assumption \ref{assufreecur}. Choose an arbitrary point $x_0\in \overline{\Omega}\cap N$. Particularly, we require $x_0\in \Omega\cap N$ if $m-n\geq 2$.  An  argument similar to the one of  (\ref{estimtevoluem}) yields a small $\eta\in (0,1)$ such that
 \[
2^{-1}t^{m-n-1}\leq \det \mathcal {A}(t,\mathbf{n})\leq 2 t^{m-n-1},\ \forall (t,\mathbf{n})\in[0,\eta)\times \mathcal {V}S(N\cap B_{2\eta}(x_0)),\tag{4.9}\label{detestimate}
\]
where $B_{2\eta}(x_0)$ is the open ball of radius $2\eta$ centered at $x_0$. And the triangle inequality implies
\[
B_{\eta}(x_0)\subset E\left( [0,\eta)\times \mathcal {V}S(N\cap B_{2\eta}(x_0))\right).
\]

Let $\phi\in C^\infty_0(M)$ be a cut-off function  with  $\phi(x)=1$ if $x\in B_{\eta/2}(x_0)$ and $\phi(x)=0$ if $x\notin B_{\eta}(x_0)$.
Choose an arbitrary constant $D\geq \sup_{x\in \Omega}r(x)$. For any $p>1$, $\alpha\in \mathbb{R}$ and small $\epsilon>0$, set
\[
J_\alpha(\epsilon):=\int_\Omega \phi^p \, r^{-(m-n)+\epsilon p}\,\log^{\alpha}\left( \frac{D}{r}\right)\,\dvol_g=\int_{\Omega\cap B_{\eta}(x_0)} \phi^p \, r^{-(m-n)+\epsilon p}\,\log^{\alpha}\left( \frac{D}{r}\right)\,\dvol_g.
\]

\begin{lemma}\label{basisestima}Let $M,N$ and $\Omega$ be as in Assumption \ref{assufreecur}. Given $D\geq \sup_{x\in \Omega}r(x)$,
for small $\epsilon>0$, we have
\begin{itemize}
\item[(i)] $J_\alpha(\epsilon)=O_\epsilon(1)$, for $\alpha<-1$;

\smallskip

\item[(ii)] there exist two positive constants $c,C$  independent of $\epsilon$ such that
\[
c\,\epsilon^{-1-\alpha}\leq J_\alpha(\epsilon)\leq C\,\epsilon^{-1-\alpha}, \text{ for }\alpha>-1;
\]
\item[(iii)] $J_\alpha(\epsilon)=\frac{p\epsilon}{\alpha+1}J_{\alpha+1}(\epsilon)+O_\epsilon(1), \text{ for }\alpha>-1$.
\end{itemize}
\end{lemma}
\begin{proof} In what follows,  we use $C_1,C_2,\ldots$ to denote the constants independent of $\epsilon$.
Set $\Psi(t):=[\log(D/t)]^{-1}$.
Then (\ref{2.0'fermico}) together with  (\ref{detestimate}) yields
\begin{align*}
J_\alpha(\epsilon)
\leq  2c_{m-n-1}\vol_{i^*g}(N\cap B_{2\eta}(x_0)) \int^\eta_0 t^{-1+\epsilon p}\Psi^{-\alpha}\left( t\right)dt
=:  C_1 \int^\eta_0 t^{-1+\epsilon p}\Psi^{-\alpha}\left( t\right)dt.\tag{4.10}\label{estimateJbeta}
\end{align*}

\smallskip

(i) Suppose $\alpha<-1$. Then (\ref{estimateJbeta}) yields
\[
0\leq J_\alpha(\epsilon)\leq C_1\eta^{\epsilon p}\int^\eta_0 t^{-1}\Psi^{-\alpha}(t)dt=\frac{C_1\eta^{\epsilon p}}{|1+\alpha|} \left[\log\left( \frac{D}{\eta} \right)\right]^{\alpha+1}<+\infty,
\]
 which means $J_\alpha(\epsilon)=O_\epsilon(1)$.

\smallskip

(ii) Suppose $\alpha>-1$. We apply the change of variables $t=D s^{1/\epsilon}$ to (\ref{estimateJbeta}) and obtain
\begin{align*}
J_\alpha(\epsilon)\leq
C_1 D^{\epsilon p}\epsilon^{-1-\alpha}\int^{(\frac{\eta}{D})^\epsilon}_0 s^{p-1}\log^{\alpha}(s^{-1})ds\leq C \epsilon^{-1-\alpha},\tag{4.11}\label{speciallogestimate}
\end{align*}
where $C$ is a positive constant independent of $\epsilon$.

Now set $N_{\eta/4}:=B_{\eta/4}(x_0)\cap N$.
We consider a $\eta/4$-tubular neighbourhood of $N_{\eta/4}$ in $\Omega$, i.e.,
\[
T^\Omega_{\eta/4}(N_{\eta/4}):=\left\{x\in \Omega: d(N_{\eta/4},x)<\eta/4\right\}.
\]
The triangle inequality implies $T^\Omega_{\eta/4}(N_{\eta/4})\subset B_{\eta/2}(x_0)\cap \Omega$. Given $x\in N_{\eta/4}$, let
\[
\mathcal {V}S_{x}^-N_{\eta/4}:=\left\{\mathbf{n}\in \mathcal {V}S_xN:\, \exists \, t_0>0 \text{ such that }\gamma_\mathbf{n}(t)\in T^\Omega_{\eta/4}(N_{\eta/4}) \text{ for any }t\in (0,t_0)\right\}.
\]
By (\ref{2.0'fermico}) and (\ref{detestimate}) again, we have
\begin{align*}
J_\alpha(\epsilon)
\geq &\int_{N_{\eta/4}}{\dvol}_{i^*g}(x)\int_{\mathcal {V}S^-_xN_{\eta/4}}d\nu_x(\mathbf{n})\int_0^{\eta/4}\phi^p(\Exp(t\mathbf{n}))\, t^{-(m-n)+\epsilon p}\,\Psi^{-\alpha}\left( t\right)\, \det \mathcal {A}(t,\mathbf{n})dt\\
\geq&\frac12\int_{N_{\eta/4}}{\dvol}_{i^*g}(x)\int_{\mathcal {V}S^-_xN_{\eta/4}}d\nu_x(\mathbf{n})\int_0^{\eta/4}  t^{-1+\epsilon p}\,\Psi^{-\alpha}\left( t\right)dt,
\end{align*}
which together with an estimate analogous to
(\ref{speciallogestimate}) yields  a constant $c>0$ independent of $\epsilon$ such that
\[
J_\alpha(\epsilon)\geq c \epsilon^{-1-\alpha}.
\]

(iii) Given  any $l \in \mathbb{R}$, (\ref{2.0'fermico}) combined with  (\ref{detestimate}) and Lemma \ref{interglemaGa} furnishes
\begin{align*}
0\leq \lim_{s\rightarrow 0^+}\int_{T_s\cap B_\eta(x_0)}r^{-(m-n)+\epsilon p}\Psi^l(r)\dvol_g\leq C_2\lim_{s\rightarrow 0^+}\int^s_0 t^{-1+\epsilon p}\log^{-l}\left( \frac{D}t  \right)dt= 0.\tag{4.12}\label{controllog}
\end{align*}
It follows from (\ref{logdervative}), the  divergence theorem, (\ref{controllog}),  (\ref{smallestimatelaplace}) and Lemma \ref{interglemaGa} that
\begin{align*}
(\alpha+1)J_\alpha(\epsilon)=&-\lim_{s\rightarrow 0}\int_{\Omega\backslash T_s} g{\left( \nabla {\left[\Psi^{-1-\alpha}\left( r \right)\right]},\phi^p(x)\, r^{1-(m-n)+\epsilon p}\nabla r\right)} \dvol_g\\
=&\lim_{s\rightarrow 0}\int_{\Omega\backslash T_s} \di\left( \phi^p\, r^{1-(m-n)+\epsilon p} \nabla r \right) \Psi^{-1-\alpha}\left( r \right)\dvol_g=I_1+I_2,\tag{4.13}\label{sec5divd}
\end{align*}
where $\ds\ds I_1:=p\int_\Omega \phi^{p-1} r^{1-(m-n)+\epsilon p}\Psi^{-1-\alpha}\left( r \right)g( \nabla\phi, \nabla r) \dvol_g$ and
\[
I_2:=(1-(m-n)+\epsilon p)\int_\Omega \phi^p r^{-(m-n)+\epsilon p}\Psi^{-1-\alpha}\left( r\right)\dvol_g
+\int_\Omega \phi^p r^{1-(m-n)+\epsilon p}\Delta r \Psi^{-1-\alpha}\left( r \right)\dvol_g.
\]
On the one hand, (\ref{2.0'fermico}) together with (\ref{detestimate}) and Lemma \ref{interglemaGa} implies
\begin{align*}
|I_1|
\leq C_3\,\eta^{\epsilon p}\int^\eta_0  \Psi^{-1-\alpha}\left( t\right)dt<\infty\Longrightarrow I_1=O_\epsilon(1).\tag{4.14}\label{5.5}
\end{align*}
On the other hand, the same argument combined with (\ref{smallestimatelaplace}) furnishes
\begin{align*}
I_2
=\epsilon pJ_{\alpha+1}(\epsilon)+\int_\Omega \phi^p r^{-(m-n)+\epsilon p}\left[ r\Delta r+1-(m-n) \right]\Psi^{-1-\alpha}\left( r \right)\dvol_g=\epsilon pJ_{\alpha+1}(\epsilon)+O_\epsilon(1).\tag{4.15}\label{lastestimateJ}
\end{align*}
Consequently, Statement (iii) follows from (\ref{sec5divd})--(\ref{lastestimateJ}).
\end{proof}

Given $p>1$,  $\beta\neq-(m-n)$ and $D\geq \sup_{x\in \Omega}r(x)$,
for small $\epsilon>0$,  set
\[
u_\epsilon(x):=\left \{
\begin{array}{lll}
\phi(x)\cdot \omega_\epsilon(x),& \text{ if }x\in  \Omega, \\
\tag{4.16}\label{uespdefine}\\
0,& \text{ if }x\notin \Omega,
\end{array}
\right.
\]
where $\phi$ is the cut-off function defined as before and
\[
\omega_\epsilon(x):=r(x)^{-\delta+\epsilon}\,\Psi^{-\theta}\left(r(x)\right), \ \Psi(t):=\left[ \log\left( \frac{D}{t} \right) \right]^{-1},\  \delta:=\frac{m-n+\beta}p,\  \frac1p<\theta<\frac2p.
\]

\begin{lemma}\label{littleestiamteu}Given $p>1$, $\beta\neq-(m-n)$ and $D\geq \sup_{x\in \Omega}r(x)$, let $\mathscr{I}[\cdot]$ be the functional defined as in  (\ref{FunctionalI}) and let  $u_\epsilon(x)$ be the function defined as in (\ref{uespdefine}). Under Assumption \ref{assufreecur},
we have
\begin{itemize}
\item[(a)] $\mathscr{I}[u_\epsilon]\leq \frac{\theta(p-1)}{2}|\delta|^{p-2}J_{p\theta-2}(\epsilon)+O_\epsilon(1)$;

\smallskip

\item[(b)] $\int_{\Omega} |\nabla u_\epsilon|^p r^{p+\beta}{\dvol}_g\leq |\delta|^p J_{p\theta}(\epsilon)+O_\epsilon(\epsilon^{1-p\theta})$.
\end{itemize}
\end{lemma}
\begin{proof}In what follows, the constants $C_1,C_2,\ldots$ are independent of $\epsilon$.
Since there exists a positive constant $C_1=C_1(p)$ such that $|a+b|^p\leq  C_1 \left( |a|^{p-1}|b|+|b|^p \right)+|a|^p$,
 we have
\begin{align*}
\int_\Omega |\nabla u_\epsilon|^pr^{p+\beta}{\dvol}_g=\int_{B_{\eta}(x_0)\cap \Omega}| \phi \nabla\omega_\epsilon+\omega_\epsilon\nabla\phi |^pr^{p+\beta}{\dvol}_g\leq I_1+I_2+I_3,\tag{4.17}\label{threesums}
\end{align*}
where $\ds\ds I_1:=C_1 \int_{B_{\eta}(x_0)\cap \Omega}\phi^{p-1} |\nabla\omega_\epsilon|^{p-1}|\nabla\phi| |\omega_\epsilon|r^{p+\beta}{\dvol}_g$, $\ds\ds I_2:=C_1\int_{B_{\eta}(x_0)\cap \Omega}|\nabla\phi|^p|\omega_\epsilon|^pr^{p+\beta}{\dvol}_g$ and
\begin{align*}
 I_3:=\int_{B_{\eta}(x_0)\cap \Omega}\phi^p |\nabla\omega_\epsilon|^pr^{p+\beta}{\dvol}_g.
\end{align*}

Now we estimate $I_1$. In view of (\ref{logdervative}), one has
\[
\nabla\omega_\epsilon=r^{-\delta+\epsilon-1}\Psi^{-\theta}\left( r \right)\left( -\delta+\epsilon-\theta \Psi\left( r \right)  \right) \nabla r.\tag{4.18}\label{derivatofw}
\]
For small $\epsilon>0$, the same argument as in (\ref{5.5}) combined with (\ref{derivatofw}) yields
\begin{align*}
0\leq I_1&\leq C_2\int_{B_{\eta}(x_0)\cap \Omega} r^{1-(m-n)+\epsilon p}\Psi^{-\theta p}\left( r \right)\left| \delta-\left(\epsilon-\theta\Psi\left( r \right)\right)  \right|^{p-1}{\dvol}_g\\
&\leq C_3\int_{B_{\eta}(x_0)\cap \Omega} r^{1-(m-n)+\epsilon p}\Psi^{-\theta p}\left( r \right) {\dvol}_g\Longrightarrow I_1=O_\epsilon(1).\tag{4.19}\label{5.7}
\end{align*}
Similarly, one has $I_2=O_\epsilon(1)$,
which together with (\ref{threesums}) and (\ref{5.7}) indicates
\[
\int_\Omega |\nabla u_\epsilon|^pr^{p+\beta}{\dvol}_g\leq I_3+O_\epsilon(1).\tag{4.20}\label{Iafrom}
\]

Now we show Statement (a).
Combing (\ref{Iafrom}),    (\ref{FunctionalI}) and (\ref{derivatofw}), we have
\begin{align*}
\mathscr{I}[u_\epsilon]\leq I_3-|\delta|^p J_{p\theta}(\epsilon)+O_\epsilon(1)=I_A+O_\epsilon(1),\tag{4.21}\label{5.8}
\end{align*}
where $\ds\ds I_A:=\int_{B_{\eta}(x_0)\cap \Omega}\phi^p r^{-(m-n)+p\epsilon}\Psi^{-\theta p}\left( r \right)\left[ \left|\delta-\zeta\right|^p-|\delta|^p \right]{\dvol}_g$ and $\zeta:=\left(\epsilon-\theta\Psi\left( r\right)\right)$. By choosing a small $\eta>0$, we can suppose $|\zeta| \ll|\delta|$.
Thus  Taylor's expansion of $f(t)=|t|^p$ furnishes
\[
|\delta-\zeta|^p-|\delta|^p\leq -p |\delta|^{p-2}\delta\zeta+\frac12p(p-1)|\delta|^{p-2}\zeta^2+C_4|\zeta|^3,
\]
which implies
\begin{align*}
I_A\leq I_{A1}+I_{A2}+I_{A3},\tag{4.22}\label{5.9}
\end{align*}
where
\begin{align*}
I_{A1}:&=-p|\delta|^{p-2}\delta \int_{B_{\eta}(x_0)} \phi^p r^{-(m-n)+\epsilon p}\Psi^{-p\theta}\left( r\right)\left(\epsilon-\theta \Psi\left(r\right)  \right){\dvol}_g,\\
I_{A2}:&=\frac12p(p-1)|\delta|^{p-2}\int_{B_{\eta}(x_0)}\phi^p r^{-(m-n)+p\epsilon}\Psi^{-\theta p}(r)\left(\epsilon-\theta\Psi\left( r\right)\right)^2{\dvol}_g,\\
I_{A3}:&=C_4\int_{B_{\eta}(x_0)}\phi^p r^{-(m-n)+\epsilon p}\Psi^{-p\theta}\left( r \right)\left|\epsilon-\theta \Psi\left(r\right)  \right|^3{\dvol}_g.
\end{align*}
Lemma \ref{basisestima}/(iii) yields
\begin{align*}
I_{A1}=-p|\delta|^{p-2}\delta\left( \epsilon J_{p\theta}(\epsilon)-\theta J_{p\theta-1}(\epsilon)  \right)= O_\epsilon(1),\tag{4.23}\label{5.11}
\end{align*}
whereas  Lemma \ref{basisestima}/(i)-(ii) implies
\begin{align*}
I_{A3}\leq C_5\epsilon^3 J_{p\theta}(\epsilon)+C_6J_{p\theta-3}(\epsilon)\leq C_7 \epsilon^{2-p\theta}+O_\epsilon(1)=O_\epsilon(1).\tag{4.24}\label{5.12}
\end{align*}
Furthermore, by using Lemma \ref{basisestima}/(iii) twice, we have
\begin{align*}
I_{A2}
&=\frac12p(p-1)|\delta|^{p-2}\left( \epsilon^2 J_{p\theta}(\epsilon)-2\epsilon\theta J_{p\theta-1}(\epsilon)+\theta^2J_{p\theta-2}(\epsilon)  \right)\\
&=\frac12 p(p-1)|\delta|^{p-2}\left( \frac{\theta}{p} J_{p\theta-2}(\epsilon)\right)+O_\epsilon(1),%\tag{4.25}\label{5.13}
\end{align*}
which together with (\ref{5.8})--(\ref{5.12}) indicates Statement (a).
On the other hand,
Statement (a) combined with
Lemma  \ref{basisestima}/(ii) furnishes
\begin{align*}
\int_{\Omega}|\nabla u_\epsilon|^pr^{p+\beta} {\dvol}_g= \mathscr{I}[u_\epsilon]+|\delta|^p J_{p\theta}(\epsilon)\leq O_\epsilon(\epsilon^{1-p\theta}) + |\delta|^p J_{p\theta}(\epsilon).
\end{align*}
Thus, Statement (b) follows.
\end{proof}

In view of  Definitions \ref{Dsdef} and \ref{Hanshukongj}, we have the following result.
\begin{corollary}\label{basissolutionfornoncaompcat}Let $M,N$ and $\Omega$ be as in Assumption \ref{assufreecur} and let  $u_\epsilon(x)$ be as in (\ref{uespdefine}). Given $p>1$, $\beta\in \mathbb{R}$ and $D\geq \sup_{x\in \Omega}r(x)$,  the following statements hold:
\begin{itemize}

\item[(a)] if  $\beta<-(m-n)$, then $u_\epsilon\in D^{1,p}(\Omega_N,r^{p+\beta})$ for small $\epsilon>0$;

\item[(b)] if $p\in (1,m-n)$, $\beta=-p$ and $(M,g)$ is flat (i.e., $\mathbf{K}_M\equiv 0$), then $u_\epsilon\in W^{1,p}_0(\Omega_N)$ for small $\epsilon>0$.

\end{itemize}
\end{corollary}
\begin{proof}[Sketch of proof] (a) For any small $\iota>0$, set $u_{\epsilon,\iota}:=  \max\{u_\epsilon-\iota,0 \}$.
Clearly, $u_{\epsilon,\iota}$ is a Lipschitz function with compact support in $\Omega_N$. Since $u_\epsilon|_N=0$, a standard argument (cf. Hebey \cite[Lemma 2.5]{H}) yields  $u_{\epsilon,\iota}\in D^{1,p}(\Omega_N,r^{p+\beta})$.  On the other hand, a direct calculation together with
Lemma \ref{littleestiamteu}/(b) yields $\|u_\epsilon\|_{D}<+\infty$, which
combined with an argument analogous to (\ref{postivefunctionconvergence})  furnishes  $u_\epsilon\in D^{1,p}(\Omega_N,r^{p+\beta})$.

(b) In this case, $W_0^{1,p}(\Omega,r^{p+\beta})$ is  exactly the standard Sobolev space $W_0^{1,p}(\Omega)$.
Since $m-n>1$, we may assume $\supp u_\epsilon \subset \overline{B_\eta(x_0)}\subset \Omega$. Clearly, $u_\epsilon\geq 0$ is unbounded but $\|u_\epsilon\|_{p,\beta}<+\infty$.  Set $u_{\epsilon,\lambda}:=\min\{u_\epsilon,\lambda\}$ for any $\lambda>0$. Since $N$ has
 $\vol_g$-measure zero, a similar argument as above yields $u_{\epsilon,\lambda}\in W_0^{1,p}(\Omega)$ and $\lim_{\lambda\rightarrow\infty}\|u_{\epsilon,\lambda}-u_\epsilon\|_{p,\beta}=0$. Hence, $u_\epsilon\in  W_0^{1,p}(\Omega)$.

  On the other hand, note that $(M,g)$ is locally Euclidean. In particular, for any $x\in M$,  $\vol_g(B_s(x))=c_m s^m$ for  $0<s< \text{the injectivity radius at $x$}$. Moreover,  the $(m-p)$-dimensional Hausdorff measure of $N$ is zero
 (cf. Burago et al. \cite[Section 5.5]{DYS}). These facts together with
  a standard but cumbersome capacity argument (cf. Kinnunen and Martio \cite[Corollary 4.14]{KM} and Kilpel\"ainen et al. \cite[Remarks 4.2/(4), Theorem 4.6]{KKM}) furnish $ W^{1,p}_0(\Omega)=W^{1,p}_0(\Omega_N)$. Therefore, $u_\epsilon\in W^{1,p}_0(\Omega_N)$.
\end{proof}

\begin{lemma}\label{sharpnessnon-compact}Let $M,N$ and $\Omega$ be as in Assumption \ref{assufreecur}. Given $p>1$, $\beta\in \mathbb{R}$ and $D\geq \sup_{x\in \Omega}r(x)$, set $\delta := (m - n + \beta)/p$.
Suppose  one of the following conditions holds:
\begin{itemize}

\item[(a)] $\beta<-(m-n)$;

\item[(b)] $p\in (1,m-n)$, $\beta=-p$ and $(M,g)$ is flat.

\end{itemize}
Additionally suppose that for some constants $A>0$, $B\geq 0$ and $\gamma>0$, there holds
\[
\int_\Omega |\nabla u|^p r^{p+\beta} {\dvol}_g\geq A\int_\Omega  {|u|^p}{r^\beta}{\dvol}_g+B\int_\Omega  {|u|^p}{r^\beta}\log^{-\gamma}\left( \frac{D}{r}  \right){\dvol}_g,\ \forall\,u\in C^\infty_0(\Omega_N).\tag{4.25}\label{***5.13}
\]
Then we have
\begin{itemize}
\item[(i)] $A\leq |\delta|^p$;

\smallskip

\item[(ii)] if $A=|\delta|^p$ and $B>0$, then $\gamma\geq 2$;

\smallskip

\item[(iii)] if $A=|\delta|^p$ and $\gamma=2$, then $B\leq \frac{p-1}{2p}|\delta|^{p-2}$,
\end{itemize}

\end{lemma}
\begin{proof}Let $u_\epsilon$ be defined as in (\ref{uespdefine}) and set
\[
R_\gamma[u_\epsilon]:=\int_\Omega  {|u_\epsilon|^p}{r^{\beta}}\Psi^\gamma\left( r \right){\dvol}_g=\int_\Omega \phi^p r^{-(m-n)+\epsilon p}\Psi^{\gamma-\theta p}\left( r \right){\dvol}_g=J_{p\theta-\gamma}(\epsilon).
\]

\smallskip

(i) By a suitable modification to the proof of Proposition \ref{reverRicinfty}, one can show that
 \begin{itemize}
 \item (\ref{***5.13}) is valid for $D^{1,p}(\Omega_N, r^{p+\beta})$ if Condition (a) holds;

 \item (\ref{***5.13}) is valid for $W^{1,p}_0(\Omega_N, r^{p+\beta})=W^{1,p}_0(\Omega_N)$ if Condition (b) holds.
 \end{itemize}
Hence, Corollary \ref{basissolutionfornoncaompcat} together with (\ref{***5.13}) yields
\begin{align*}
A\leq \frac{\int_\Omega |\nabla u_\epsilon|^pr^{p+\beta}{\dvol}_g}{\int_\Omega  {|u_\epsilon|^p}{r^\beta} {\dvol}_g}=\frac{\int_{\Omega} |\nabla u_\epsilon|^pr^{p+\beta}{\dvol}_g}{J_{p\theta}(\epsilon)},
\end{align*}
which combined with Lemma \ref{littleestiamteu}/(b) and Lemma \ref{basisestima}/(ii) yields
\begin{align*}
A\leq |\delta|^p+\frac{O_\epsilon(\epsilon^{1-p\theta})}{J_{p\theta}(\epsilon)}\leq |\delta|^p+\frac{O_\epsilon(\epsilon^{1-p\theta})}{c\epsilon^{-1-p\theta}}\rightarrow |\delta|^p, \text{ as }\epsilon\rightarrow 0.
\end{align*}
Hence, Statement (i) follows.

\smallskip

(ii) If the assertion did not hold, then $-1<p\theta-\gamma$ and hence, Lemma \ref{littleestiamteu}/(a) combined with Lemma \ref{basisestima}/(ii) would furnish
\begin{align*}
0<B\leq \frac{\mathscr{I}[u_\epsilon]}{R_\gamma[u_\epsilon]}\leq \frac{\frac{\theta(p-1)}{2}|\delta|^{p-2}J_{p\theta-2}(\epsilon)+O_\epsilon(1)}{J_{p\theta-\gamma}(\epsilon)}\leq \frac{C\epsilon^{1-p\theta}+O_\epsilon(1)}{c\epsilon^{-1-p\theta+\gamma}}\rightarrow0, \text{ as }\epsilon\rightarrow 0,
\end{align*}
which is a contradiction. Hence, $\gamma\geq 2$.

\smallskip

(iii) If $A=|\delta|^p$ and $\gamma=2$, Lemma  \ref{basisestima}/(ii)
yields  $J_{p\theta-2}(\epsilon)=O_\epsilon(\epsilon^{1-p\theta})$, which
together with
Lemma \ref{littleestiamteu}/(a) furnishes
\begin{align*}
B\leq \frac{\mathscr{I}[u_\epsilon]}{R_2[u_\epsilon]}\leq \frac{\frac12\theta(p-1)|\delta|^{p-2}J_{p\theta-2}(\epsilon)+O_\epsilon(1)}{J_{p\theta-2}(\epsilon)}\rightarrow \frac12\theta(p-1)|\delta|^{p-2}, \text{ as }\epsilon\rightarrow 0.
\end{align*}
The proof is completed by letting $\theta\rightarrow (1/p)^+$.
\end{proof}

\begin{proof}[Proof of Theorem \ref{non-compactoptimalmainthe}]
Statement (a') (resp., Statement (b')) is a direct consequence of  Theorem \ref{keycompactHardy}/(a) (resp., Theorem \ref{mainnon-compacttheroem}) and Lemma \ref{sharpnessnon-compact}.
For Statement (c'), the same argument as in the proof of Proposition \ref{reverRicinfty} furnishes that both (\ref{strongimproveHardy}) and (\ref{strongimproveHardy2}) remain valid for $u\in W^{1,p}_0(\Omega_N,r^{p+\beta})$. Hence, they also hold for $u\in C^\infty_0(\Omega,N)$  due to
Lemma \ref{continuouslemma}, which  furnish respectively
\begin{align*}
\left|\frac{\beta+ k }{p}\right|^p&\leq\inf_{u\in C^\infty_0(\Omega, N)\backslash\{0\}}\frac{\int_{{\Omega}} |\nabla u|^pr^{p+\beta}{\dvol}_g}{\int_{{\Omega}} |u|^pr^\beta {\dvol}_g},\\%\tag{1.2}\label{firssharpconsnta}\\
\frac{p-1}{2p}\left|\frac{\beta+ k }{p}\right|^{p-2}&\leq\inf_{u\in C^\infty_0(\Omega, N)\backslash\{0\}}\frac{\int_{{\Omega}} |\nabla u|^pr^{p+\beta}{\dvol}_g-\left|\frac{\beta+ k }{p}\right|^p\int_{{\Omega}} |u|^pr^\beta {\dvol}_g}{\int_\Omega  {|u|^p}{r^\beta}\log^{-2}\left( \frac{D}{r} \right){\dvol}_g}.
\end{align*}
The reverse inequalities follow from (\ref{cosntantsharpnoimpH}), (\ref{strongnessconstantH}) and
 $C^\infty_0(\Omega_N)\subset C^\infty_0(\Omega,N)$.
\end{proof}

By a suitable modification to Barbatis et al. \cite[Theorem B]{BFT}, one can get another
  improved version of the weighed Hardy inequality.
\begin{proposition}
Let $(M,g)$, $N$ and $\Omega$ be as in Assumption \ref{assufreecur} and satisfy Condition (\textbf{C}).
For any $D> \sup_{x\in \Omega}r(x)$, $1\leq q<p$, $\gamma>1+q/p$ and $\beta<-(m-n)$, there exists a constant $c>0$ such that
\[
\int_{{\Omega}} |\nabla u|^pr^{p+\beta} {\dvol}_g\geq \left|\frac{\beta+ m-n }{p}\right|^p\int_{{\Omega}}  {|u|^p}{r^\beta} {\dvol}_g+ c\left( \int_\Omega |\nabla u|^q \, r^{\frac{(m-n)(q-p)+q(\beta+p)}p} \log^{-\gamma}\left(\frac{D}r\right)\dvol_g \right)^{\frac{p}q},
\]
for any $u\in C^\infty_0(\Omega_N)$. In particular,   $\gamma$ cannot be smaller than $1+q/p$.
\end{proposition}

\section{Hardy inequalities in the flat case}\label{flathardy}

In this section, we present two non-weighted Hardy inequalities
in the case when $M$ is a flat manifold (i.e., $\mathbf{K}_M\equiv0$) and $N$ is a minimal submanifold (i.e., $H\equiv0$), in which case the requirement $p>k$ in Theorems \ref{keycompactHardy} and \ref{non-compactoptimalmainthe} can be eliminated.

\begin{lemma}\label{flatcase}
Let $(M,g)$ be an $m$-dimensional complete flat Riemannian manifold and let $i:N\hookrightarrow M$ be an $n$-dimensional  minimal closed submanifold. For any $\mathbf{n}\in \mathcal {V}SN$,
\[
\det \mathcal {A}(t,\mathbf{n})=t^{m-n-1},\quad \Delta r|_{(t,\mathbf{n})}=\frac{m-n-1}{t}, \quad  \text{for } 0<t<c_\mathcal {V}(\mathbf{n}).
\]
\end{lemma}
\begin{proof}Let $(t,\mathbf{n})$ denote Fermi coordinates, let $(e_\mathbbm{a})_{\mathbbm{a}=\alpha,\mathfrak{g}}$ be as in (\ref{defofea}) and set $J_\mathbbm{a}(t):=\tau_{t;\mathbf{n}}\mathcal {A}(t,\mathbf{n})e_\mathbbm{a}$. According to Chavel \cite[p.\,321]{IC} (by choosing $\lambda=0$ and $\mathfrak{s}_K(t)=t$), we have $J_\alpha(t)=
\tau_{t;\mathbf{n}}e_\alpha$ and $J_\mathfrak{g}(t)=
t\,\tau_{t;\mathbf{n}}e_\mathfrak{g}$,
which implies $\det \mathcal {A}(t,\mathbf{n})= t^{m-n-1}$. We conclude  the proof by (\ref{lapacianmeaning}).
\end{proof}

\begin{theorem}\label{flatcaseThe1112}
Let $(M,g)$ be an $m$-dimensional complete  flat Riemannian manifold, let $i:N\hookrightarrow M$ be an $n$-dimensional  minimal closed submanifold, and let $\Omega$ be a natural domain in $M$.

\begin{itemize}
\item[(i)]
For any $1<p\neq (m-n)$, there always holds
 \begin{align*}
\int_{{\Omega}} |\nabla u|^p{\dvol}_g\geq \left|\frac{m-n-p }{p}\right|^p\int_{{\Omega}} \frac{|u|^p}{r^p} {\dvol}_g, \ \forall\, u\in C^\infty_0({\Omega, N}).\tag{5.1}\label{firsthardpructreflatcase}
\end{align*}
 In particular, if $N$ is compact, then (\ref{firsthardpructreflatcase}) is optimal in the following sense
 \[
\left|\frac{m-n-p }{p}\right|^p=\inf_{u\in C^\infty_0(\Omega, N)\backslash\{0\}}\frac{\int_{{\Omega}} |\nabla u|^p {\dvol}_g}{\int_{{\Omega}} \frac{|u|^p}{r^p} {\dvol}_g}.%\tag{5.2}\label{firssharpconsntaflat}
\]

\smallskip

\item[(ii)] Additionally suppose  $\sup_{x\in \Omega}r(x)<+\infty$.
Thus,
for any $1<p\neq (m-n)$,
there always holds
 \begin{align*}
\int_{{\Omega}} |\nabla u|^p{\dvol}_g\geq \left|\frac{m-n-p }{p}\right|^p\int_{{\Omega}} \frac{|u|^p}{r^p} {\dvol}_g, \ \forall\, u\in C^\infty_0({\Omega, N}).\tag{5.2}\label{firsthardpructreflatcase2}
\end{align*}
 In particular,  (\ref{firsthardpructreflatcase2}) is optimal in the following sense
 \[
\left|\frac{m-n-p }{p}\right|^p=\inf_{u\in C^\infty_0(\Omega, N)\backslash\{0\}}\frac{\int_{{\Omega}} |\nabla u|^p {\dvol}_g}{\int_{{\Omega}} \frac{|u|^p}{r^p} {\dvol}_g}.\tag{5.3}\label{firssharpconsntaflat1}
\]
Moreover, there
exists a constant $\mathcal {T}=\mathcal {T}(p,m-n)> 1$ such that for any $D\geq \mathcal {T}\sup_{x\in \Omega}r(x)$,
 \begin{align*}
\int_{{\Omega}} |\nabla u|^p {\dvol}_g\geq& \left|\frac{m-n-p }{p}\right|^p\int_{{\Omega}}  \frac{|u|^p}{r^p} {\dvol}_g+\frac{p-1}{2p}\left|\frac{m-n-p }{p}\right|^{p-2}\int_\Omega  \frac{|u|^p}{r^p}\log^{-2}\left( \frac{D}{r} \right){\dvol}_g,\tag{5.4}\label{strongimproveHardyfaltacase}
\end{align*}
for any $u\in C^\infty_0({\Omega, N})$. In particular, (\ref{strongimproveHardyfaltacase}) is optimal in the sense of (\ref{firssharpconsntaflat1}) as well as
\begin{align*}
\frac{p-1}{2p}\left|\frac{m-n-p }{p}\right|^{p-2}=\inf_{u\in C^\infty_0(\Omega, N)\backslash\{0\}}\frac{\int_{{\Omega}} |\nabla u|^p{\dvol}_g-\left|\frac{m-n-p }{p}\right|^p\int_{{\Omega}} \frac{|u|^p}{r^p} {\dvol}_g}{\int_\Omega  \frac{|u|^p}{r^p}\log^{-2}\left( \frac{D}{r} \right){\dvol}_g}.\tag{5.5}\label{firssharpconsntaflat2}
\end{align*}
\end{itemize}
\end{theorem}
\begin{proof}[Sketch of proof]On account of Theorems \ref{keycompactHardy}$\&$\ref{non-compactoptimalmainthe}, it suffices to show the theorem when  $p<(m-n)$.  Statement (i) follows from the same proof of Theorem \ref{keycompactHardy} together with Lemma \ref{flatcase}.

Now we prove Statement (ii). Firstly, since $C^\infty_0(\Omega_N)\subset C_0^\infty(\Omega,N)$, it follows from (\ref{firsthardpructreflatcase}) and Lemma \ref{sharpnessnon-compact} that (\ref{firsthardpructreflatcase2}) is optimal. Secondly,
by  an easy modification to the last part of the proof of Theorem \ref{mainnon-compacttheroem} (see (a),(b) in Barbatis et al. \cite[p.2182]{BFT}), one can obtain a constant $\mathcal {T}=\mathcal {T}(p,m-n)> 1$ such that for any $D\geq \mathcal {T}\sup_{x\in \Omega}r(x)$, there holds
 \begin{align*}
\int_{{\Omega}} |\nabla u|^p {\dvol}_g\geq  \left|\frac{m-n-p }{p}\right|^p\int_{{\Omega}}  \frac{|u|^p}{r^p} {\dvol}_g+\frac{p-1}{2p}\left|\frac{m-n-p }{p}\right|^{p-2}\int_\Omega  \frac{|u|^p}{r^p}\log^{-2}\left( \frac{D}{r} \right){\dvol}_g,\tag{5.6}\label{flatstrongimproveHardy2}
\end{align*}
for any $u\in C^\infty_0({\Omega_N})$. On the one hand, Fatou's lemma indicates that (\ref{flatstrongimproveHardy2}) holds for $W^{1,p}_0(\Omega_N)$, which together with Lemma \ref{continuouslemma} yields (\ref{strongimproveHardyfaltacase}).
On the other hand, Lemma \ref{sharpnessnon-compact} implies that  (\ref{flatstrongimproveHardy2}) is optimal with respect to $C^\infty_0(\Omega_N)$.    Since $C^\infty_0(\Omega_N)\subset C_0^\infty(\Omega,N)$, a similar argument as in the proof of Theorem \ref{non-compactoptimalmainthe}/(c') yields
 the sharpness of (\ref{strongimproveHardyfaltacase}).
\end{proof}

\section{Examples}\label{examplesection}
In this section, we establish several Hardy inequalities on cylinders, hemispheres and tori, which verify the validity of Theorems  \ref{keycompactHardy}, \ref{non-compactoptimalmainthe} and \ref{flatcaseThe1112}.
In order to do this, we need two lemmas. The first one follows from
a direct computation (cf. D'Ambrosio \cite[p.\,460f]{D}).
\begin{lemma}\label{standardcalculation} Let $(M,g)$ be a complete Riemannian manifold. Suppose that $\rho$ is a  non-negative function such that $\rho>0$ a.e. and $\nabla \rho$ exists a.e.. Given any $p\in [2,+\infty)$, $\alpha\in \mathbb{R}\backslash\{0\}$ and
 $u\in C^\infty_0(M)$, there holds
\[
|\nabla u|^p\geq |\gamma|^p\frac{|u|^p}{\rho^p}|\nabla\rho|^p+\left(\frac{p-1}{p}\right)^{p-1}g{\left( \nabla |v|^p, |\nabla \rho^\alpha|^{p-2}\nabla \rho^\alpha \right)}+\frac{2}{p}|\gamma|^{p-2} \rho^{(\alpha-1)(p-1)+1}|\nabla \rho|^{p-2}\left| \nabla |v|^{\frac{p}2} \right|^2,
\]
where $v:=\rho^{-\gamma}u$ and $\gamma:=\alpha\frac{p-1}{p}$.
\end{lemma}

The second lemma can be proved by the same method as employed in Barbatis et al. \cite[Lemma 3.2]{BFT}.
\begin{lemma}\label{basisinequaltyfronon-compact}Let $(M,g)$, $N$ and $\Omega$ be as in Assumption \ref{assufreecur}.
Given $p\in(1,+\infty)$, $s\in \mathbb{R}$ and $D\geq \sup_{x\in \Omega}r(x)$, we have
\begin{align*}
 &\left(\frac{|s-1|}{p}\right)^{p-1}\int_\Omega \frac{|f|^p}{r^{m-n }} |r\Delta r+1-(m-n) |\,\log^{1-s}\left( \frac{D}{r} \right){\dvol}_g
+\int_\Omega \frac{|\nabla f|^p}{ r^{m-n-p}}  \log^{p-s}\left( \frac{D}{r} \right){\dvol}_g\\
\geq&\left(\frac{|s-1|}{p}\right)^p \int_\Omega \frac{|f|^p }{r^{m-n } }\log^{-s}\left( \frac{D}{r } \right){\dvol}_g,\ \forall\,f\in C^\infty_0(\Omega_N).
\end{align*}
\end{lemma}

Now we present the Hardy inequalities on cylinders.

\begin{example}\label{exmpale1} Given $n\in \mathbb{N}$ with $n\geq 1$,
let $M=\mathbb{R}\times \mathbb{S}^{n}$ be a cylinder equipped with the product metric.  Denote by
 $s_0$ (resp., $w_0$)   a fixed point  in $\mathbb{R}$ (resp., $\mathbb{S}^{n}$).

 \smallskip

\begin{itemize}

  \item[(a)] If $N=\{s_0\}\times\mathbb{S}^{n}$, then for any  $p>1$,
\begin{align*}
 \int_M |\nabla u|^p \dvol_g\geq \left( \frac{p-1}p \right)^p \int_M\frac{|u|^p}{r^p}\dvol_g, \ \forall\,u\in  C^\infty_0(M,N).\tag{6.1}\label{exampleforcompactcase}
 \end{align*}

\item[(b)] If $N=\mathbb{R}\times \{w_0\}$, then for any $p>n$,
 \begin{align*}
 \int_M |\nabla u|^p \dvol_g\geq  \left( \frac{p-n}{p} \right)^p\int_M\frac{|u|^p}{r^p}\dvol_g, \ \forall\,u\in  C^\infty_0(M,N).\tag{6.2}\label{firstHarydsubmsni}
 \end{align*}
 In particular, provided $n=1$ and $p=2$, then for any $D>\pi$, there holds
 \begin{align*}
\int_M|\nabla u|^2 \dvol_g\geq \frac14\int_M \frac{|u|^2}{r^2}\dvol_g+\frac14\int_M \frac{u^2 }{r^2 }\log^{-2}\left( \frac{D}{r } \right){\dvol}_g, \ \forall\,u\in   C^\infty_0(M,N).\tag{6.3}\label{longHardyineqex}
\end{align*}

\end{itemize}

\smallskip

Note that $M$ is of non-negative sectional curvature and $N$ is minimal in both (a) and (b) cases. Hence, (\ref{exampleforcompactcase}) follows from Theorem \ref{keycompactHardy} while (\ref{firstHarydsubmsni})--(\ref{longHardyineqex}) follow from Theorem \ref{non-compactoptimalmainthe} immediately. However, in the sequel, we prove these inequalities by another approach.
Let $(x^i)=(s,w^l)$ denote a local coordinate system of $M$, where $s$ (resp., $w^l$) denotes the coordinate of $\mathbb{R}$ (resp., local coordinates of $\mathbb{S}^{n}$).

\smallskip

(a) Without loss of generality, we assume $N=\{0\}\times\mathbb{S}^{n}$.
  Thus, for any $x=(s,w^l)$, we have $r(x)=|s|$ and hence, $\Delta r=0$ in $M\backslash N$.

Suppose  $p\in (1,2)$. Proceeding as in the proof of Lemma \ref{refinedsequence}/(a), one can easily get
 \[
 \int_M |\nabla u|^p \dvol_g\geq \left( \frac{p-1}p \right)^p \int_M\frac{|u|^p}{r^p}\dvol_g,\ u\in C^\infty_0({M\backslash N}).
 \] Then  (\ref{exampleforcompactcase}) follows from the same argument as in the proof of Proposition \ref{reverRicinfty}.

Suppose $p\geq 2$. Although
the method used above still work in this case, we prefer to prove (\ref{exampleforcompactcase}) by a direct calculation.
Given
$u\in C^\infty_0(M,N)$, Lemma \ref{standardcalculation} ($\rho:=r$ and $\alpha:=1$) yields
\begin{align*}
|\nabla u|^p\geq \left( \frac{p-1}{p} \right)^p \frac{|u|^p}{{{r}}^p}+\left( \frac{p-1}{p} \right)^{p-1}   g( \nabla |v|^p,\nabla{{r}})   \text{ in }{M\backslash N},\tag{6.4}\label{inequalityseefirst}
\end{align*}
where $v=u r^{\frac{1-p}p}$.
 On the other side, Taylor's expansion (w.r.t. the coordinate $s$) furnishes
  \[
   \left|u(x^i)\right|=\left|u(s,w^l)\right|= \left|s\,{\partial_s u} (0,w^l)+o(s)\right|= \left|  r \lim_{\eta\rightarrow 0}g\left(\nabla r|_{(\eta,w^l)}, \nabla u|_{(\eta,w^l)}\right)+o(r)\right|.
   \]
Then an easy computation yields
 a constant $C>0$ and a small $\epsilon>0$ such that
 \[
|g(\nabla u,\nabla r)|(x)\leq C,\  |u|^p(x)\leq Cr^p(x),\ |v|^p(x)\leq Cr(x),\  \left|g(\nabla |v|^p,\nabla r)\right|(x)\leq C,\ \text{ if }r(x)<\epsilon,%,\tag{6.5}\label{cylinedercompest}
 \]
which implies
\begin{align*}
\lim_{\eta\rightarrow0^+}\left[\left|\int_{T_\eta} g(\nabla |v|^p,\nabla r)\dvol_g\right|+ \left|\int_{\partial T_\eta}|v|^p dA\right|\right]\leq \lim_{\eta\rightarrow0^+}C \vol_g\left( T_\eta \right)+ 2 c_{n}\lim_{\eta\rightarrow0^+}\max_{\partial T_\eta} |v|^p=0.\tag{6.5}\label{divegencelemmaonSthird}
\end{align*}
Since $\Delta r=0$, the divergence theorem together with  (\ref{divegencelemmaonSthird})  yields
\begin{align*}
&\int_{M} g(\nabla |v|^p,\nabla r)\dvol_g=\lim_{\eta\rightarrow0^+}\left(\int_{M\backslash T_\eta} \di\left(|v|^p \nabla r\right) \dvol_g+\int_{T_\eta}g(\nabla |v|^p,\nabla r)\dvol_g\right)\\
=&\lim_{\eta\rightarrow0^+}\int_{\partial  T_\eta }|v|^p  dA+\lim_{\eta\rightarrow0^+}\int_{T_\eta } g(\nabla |v|^p,\nabla r)\dvol_g=0,
\end{align*}
which combined with (\ref{inequalityseefirst}) furnishes (\ref{exampleforcompactcase}).

(b) In this case, for any $x=(s,w)$,
we have $r(x)=d_{\mathbb{S}^{n}}(w,w_0)\leq \pi$,
 where $d_{\mathbb{S}^{n}}$ is the distance on $\mathbb{S}^{n}$. This fact together with a direct calculation in local coordinates then yields $\Delta r=(n-1)\cot r$ for $0<r<\pi$. If $p\in (1,2)$, the proof is almost the same as above. Alternatively, suppose $p\geq 2$. Since $p>n$,
\[
\Delta_p r^\alpha=\alpha^{p-1}\frac{(n-1)}{r^{n}}\left( r\cot r-1 \right)\leq 0.\tag{6.6}\label{cosumbmanifoldsphere}
\]
where $\alpha:= {(p-n)}/{(p-1)}$.
For any $u\in C^\infty_0(M,N)$,  Taylor's  expansion in Fermi coordinates yields a constant $C>0$ and $\epsilon>0$ such that
\[
|g(\nabla u,\nabla r)|(x)\leq C,\ |u|^p(x)\leq Cr^p(x),\ |v|^p(x)\leq Cr^{n}(x),\  \left|g(\nabla |v|^p,\nabla r )\right|(x)\leq C r^{n-1}(x), \text{ if } r(x)<\epsilon,
\]
where $v=ur^{-\alpha\frac{p-1}{p}}=u r^{ \frac{n-p}p}$. Since $v$ is compactly supported, a similar argument as above combined with the divergence theorem and (\ref{cosumbmanifoldsphere}) furnishes
\begin{align*}
&\int_M g{\left( \nabla |v|^p, |\nabla r^\alpha|^{p-2}\nabla r^\alpha \right)}\dvol_g=\\
&\lim_{\eta\rightarrow 0^+}\left(  \int_{\partial T_\eta} g\left(|v|^p |\nabla r^\alpha|^{p-2}\nabla r^\alpha, \nabla r\right)  d A- \int_{M\backslash T_\eta}|v|^p \Delta_p r^\alpha   \dvol_g+\int_{T_\eta} g{\left( \nabla |v|^p, |\nabla r^\alpha|^{p-2}\nabla r^\alpha \right)}\dvol_g \right)\geq 0,
\end{align*}
which together with Lemma \ref{standardcalculation} ($\rho:=r$) yields (\ref{firstHarydsubmsni}).

In order to prove (\ref{longHardyineqex}), for any $u\in C^\infty_0({M\backslash N})$, set $v:=ur^{-1/2}\in C^\infty_0({M\backslash N})$. Since $\Delta r=0$ (i.e., $n=1$),
  Lemma \ref{basisinequaltyfronon-compact}  ($s:=2=p$ and $f:=   v$) yields
\begin{align*}
\int_M |\nabla |v| |^2 r{\dvol}_g=\int_M |\nabla v |^2 r{\dvol}_g\geq\frac14 \int_M \frac{|v|^2 }{r }\log^{-2}\left( \frac{D}{r } \right){\dvol}_g=\frac14 \int_M \frac{u^2 }{r^2 }\log^{-2}\left( \frac{D}{r } \right){\dvol}_g,
\end{align*}
which together with Lemma \ref{standardcalculation} ($\alpha:=1$ and $\rho:=r$) and the divergence theorem furnishes
 \begin{align*}
\int_M|\nabla u|^2 \dvol_g\geq \frac14\int_M \frac{|u|^2}{r^2}\dvol_g+\frac14\int_M \frac{u^2 }{r^2 }\log^{-2}\left( \frac{D}{r } \right){\dvol}_g.\tag{6.7}\label{needsfatou}
\end{align*}
Recall
 $C^\infty_0(M,N)\subset W_0^{1,p}({M\backslash N},r^0)$ (cf.\,Lemma \ref{continuouslemma}).
Thus, (\ref{longHardyineqex}) follows from (\ref{needsfatou}) and the same argument as in the proof of Proposition \ref{reverRicinfty}.
 \end{example}

Now we turn to consider the Hardy inequalities on hemispheres.

\begin{example}\label{sphereexample}
Given $n\geq 2$, let $M:=\mathbb{S}^n\subset \mathbb{R}^{n+1}$. Denote by $\Omega$ the northern hemisphere and set $N:=\partial \Omega\cong\mathbb{S}^{n-1}$. Thus, for any $p>1$, there holds
\begin{align*}
\int_{\Omega}|\nabla u|^p\dvol_g\geq \left( \frac{p-1}{p} \right)^p\int_{\Omega}  \frac{|u|^p}{{r}^p} \dvol_g, \ \forall\,u\in C_0^\infty(\Omega,N). \tag{6.8}\label{finhardyonsphere}
\end{align*}

\smallskip
In fact, let $(\varrho,y)$ denote the polar coordinates about the north pole $o$. Thus, for any $x=(\varrho,y)\in \Omega$, we have $r(x)=\pi/2-\varrho$, which implies
\begin{align*}
\Delta r=-\Delta \varrho=-(n-1)\cot\rho=-(n-1)\tan r< 0.
\end{align*}
Consequently, (\ref{finhardyonsphere}) can be proved by the same method as employed in  Example \ref{exmpale1}/(a).
 Alternatively, according to do Carmo and Warner \cite[Theorem 1.1]{CW}, the boundary of a hemisphere is always  weakly mean convex (factually, totally geodesic). Thus, (\ref{finhardyonsphere}) follows from (\ref{sendhardnonprct}) immediately.
\end{example}

We end this section by investigating Hardy inequalities on flat tori.

\begin{example}\label{flattoriexample}
Given $m,n\in \mathbb{N}$ with $m\geq 2$ and $0\leq n\leq m-1$, let $M:=\mathbb{T}^m=\mathbb{S}^1\times \cdots \times \mathbb{S}^1$ be an $m$-dimensional flat torus and let $N:= \{w_1\}\times \cdots \times \{w_{m-n}\}\times \mathbb{T}^n\subset M$, where $w_l$, $1\leq l\leq m-n$, are fixed points in $\mathbb{S}^1$.
 Then for any $1<p\neq (m-n)$, there holds
\begin{align*}
\int_{M}|\nabla u|^p\dvol_g\geq \left| \frac{p-(m-n)}{p} \right|^p\int_{M}  \frac{|u|^p}{{r}^p} \dvol_g, \ \forall\,u\in C_0^\infty(M,N). \tag{6.9}\label{flatoruex}
\end{align*}

In fact, let $x=(x^i)=(x^1,\ldots,x^m)$, $x^i\in (-\pi,\pi)$ be the natural local coordinates of $M$ such that $x^l(w_{l})=0$ for $1\leq l\leq m-n$. Thus $r(x)=\sqrt{\sum_{l=1}^{m-n} (x^l)^2}$ and hence,
$r\Delta r=m-n-1$.
Owing to this fact, by choosing $\alpha=\frac{p-(m-n)}{p-1}$, the proof is similar to that of Example \ref{exmpale1}/(b).
By contrast, since $N$ is a minimal (actually, totally geodesic) submanifold of $M$,
 (\ref{flatoruex}) is a direct consequence of Theorem \ref{flatcaseThe1112}/(i).
\end{example}

\appendix

\section{Weighted Sobolev spaces}\label{Soblevspace}

In this section, we study the properties of weighted Sobolev spaces defined by Definition \ref{Hanshukongj}. In what follows, we always assume
\begin{itemize}
\item $(M,g)$, $N$ and $\Omega$ satisfy Assumption \ref{assufreecur1};

\item $p\in (1,+\infty)$ and $\beta\in \mathbb{R}$ satisfy    $p+\beta>-(m-n)$.
\end{itemize}

A standard argument together with   Lemma \ref{centerzeroinfite} yields the following result.
\begin{lemma}\label{nullmeaure}
Let ${\dmu}:=r^{p+\beta}{\dvol}_g$. Then ${\dmu}$ is a $\sigma$-finite measure on $M$. Moreover,
for any Borel set $E$,  $\mu(E)=0$ if and only if $\vol_g(E)=0$. In particular, $\mu(N)=0$.
\end{lemma}

In the sequel, we prefer to use ${\dmu}$ rather than $\dvol_g$ to investigate weighted Sobolev spaces. Let $U\subset M$ be an open set.
 For any $s\in (1,\infty)$, we define
  $L^s({U},r^{p+\beta})$ (resp., $L^s(T{U},r^{p+\beta})$) as the completion of $C_0({U})$ (resp., $\Gamma_0(T{U})$, i.e., the space of  continuous tangent vector fields with compact support in ${U}$) under the norm
\begin{align*}
[u]_{s,\mu}:=\left( \int_{{{U}}}|u|^s{\dmu}\right)^{\frac1s}, \ \ \  \left(\text{ resp., }  [X]_{s,\mu}:=\left( \int_{{{U}}}|X|^s{\dmu}\right)^{\frac1s} \right).
\end{align*}

The standard theory yields the following result.
\begin{theorem}\label{LPreflexive}
Both $L^s({U},r^{p+\beta})$ and $L^s(T{U},r^{p+\beta})$ are reflexive. Hence, if a sequence $(u_i)$ (resp., $(X_i)$) is bounded in
$L^s({U},r^{p+\beta})$ (resp., $L^s(T{U},r^{p+\beta})$), then there exists a subsequence $(u_{i_k})$ (resp., $(X_{i_k})$) and $u\in L^s({U},r^{p+\beta})$ (resp., $X\in L^s(T{U},r^{p+\beta})$ ) such that $u_{i_k}\rightarrow u$
weakly in $L^s({U},r^{p+\beta})$ (resp., $X_{i_k}\rightarrow X$
weakly in $L^s(T{U},r^{p+\beta})$) as $k\rightarrow \infty$.
\end{theorem}

Let ${W^{1,p}}({{U}},{r}^{p+\beta})$ be  as in Definition \ref{Hanshukongj}. Thus,  for any $u\in C^\infty({U})$, we have $\|u\|^p_{p,\beta}=[u]^p_{p,\mu}+[\nabla u]^p_{p,\mu}$.
\begin{remark}\label{smoothdivege}It is easy to check that the weak gradient in the sense of $W^{1,p}({U},r^{p+\beta})$ satisfies the following properties:
\begin{itemize}

\item[(a)] $\nabla(\lambda  u+\zeta  v)=\lambda\nabla  u+\zeta\nabla  v$ for any $\lambda,\zeta\in \mathbb{R}$ and $u,v\in W^{1,p}({U},r^{p+\beta})$;

\item[(b)] for any smooth vector field $X$ with compact support in ${U}\backslash N$, one has
\[
\int_{{{U}}} g( \nabla u ,X) {\dmu}=-\int_{{{U}}}u \di_\mu X {\dmu}.
\]
\end{itemize}
\end{remark}

In the following, $U$ is either $M$ or $\Omega_N$.
Since $p+\beta>-(m-n)$, there always exists some $q>1$ such that $q(p+\beta)>-(m-n)$. Then the same proof as in Meng et al. \cite[Lemma A.1]{MWZ} (by choosing $\rho:=r$ and ${\dmu}:=d\vol_g$) yields
\begin{lemma}\label{lipshccompact}
For
any globally Lipschitz function $u$  with compact support in $U$, we have
$u\in W_0^{1,p}(U,r^{p+\beta})$.
\end{lemma}

Due to Lemma \ref{lipshccompact}, the following result can be proved in the same way as in Hebey \cite[Theorem 2.7]{H}.
\begin{theorem}\label{MandM0isequal}
$W^{1,p}(M,r^{p+\beta})=W_0^{1,p}(M,r^{p+\beta})$.
\end{theorem}

Theorem \ref{MandM0isequal} combined with the same proof of Zhao \cite[Lemma B.4]{Z3} yields the following result.
\begin{corollary}\label{compwp}
If $u\in W^{1,p}(M,r^{p+\beta})$ with compact support in ${\Omega_N}$, then $u|_{\Omega_N}\in W_0^{1,p}({\Omega_N},r^{p+\beta})$.
\end{corollary}

In order to show the reflexivity of $W^{1,p}(M,r^{p+\beta})$, we
recall Mazur's lemma (cf. Renardy and Rogers \cite[Lemma 10.19]{RR}).

\begin{lemma}[Mazur's lemma]\label{Marzurlemma} Assume that $X$ is a Banach space and that
$x_i\rightarrow x$ weakly in $X$ as $i \rightarrow \infty$. Then there exists a sequence of convex
combinations
\[
\tilde{x}_i=\sum_{j=i}^{m_i}a_{i,j}x_j,\ a_{i,j}\geq 0,\ \sum_{j=i}^{m_i}a_{i,j}=1,
\]
such that $\tilde{x}_i\rightarrow x$ strongly in $X$ as $i\rightarrow \infty$.
\end{lemma}

\begin{theorem}\label{weakconvergece}
$W^{1,p}({M},r^{p+\beta})$ is reflexive. That is, given  a bounded sequence $(u_i)$  in $W^{1,p}({M},r^{p+\beta})$, there exists a subsequence $(u_{i_k})$ and $u\in  W^{1,p}({M},r^{p+\beta})$ such that $u_{i_k}\rightarrow u$ weakly in $L^p({M},r^{p+\beta})$  and $\nabla u_{i_k}\rightarrow \nabla u$ weakly in $L^p(T{M},r^{p+\beta})$ as $k\rightarrow \infty$.
\end{theorem}

\begin{proof}
Since $(u_i)$ is a bounded sequence in $W^{1,p}({M},r^{p+\beta})$,  Theorem \ref{LPreflexive} yields    a subsequence $(u_{i_k})$ and $u\in L^p({M},r^{p+\beta})$ and a vector field $X\in  L^p(T{M},r^{p+\beta})$ such that $u_{i_k}\rightarrow u$ weakly in $L^p ({M},r^{p+\beta})$ and $\nabla u_{i_k}\rightarrow X$ weakly in $L^p(T {M},r^{p+\beta})$ as $k\rightarrow \infty$.
Consider the Banach space $L^p({M},r^{p+\beta})\times L^p(T{M},r^{p+\beta})$ endowed with the product metric. Obviously,
\[
(u_{i_k},\nabla u_{i_k})\rightarrow (u,X) \text{ weakly in }L^p({M},r^{p+\beta})\times L^p(T{M},r^{p+\beta}).
\]
Thus, Lemma \ref{Marzurlemma} yields a   sequence of convex combinations such that
\[
\sum_{k=l}^{m_l}a_{l,k}(u_{i_k},\nabla u_{i_k})\rightarrow (u,X) \text{ strongly in }L^p({M},r^{p+\beta})\times L^p(T{M},r^{p+\beta}), \text{ as }l\rightarrow \infty.
\]
Set $v_l:=\sum_{k=l}^{m_l}a_{l,k}u_{i_k}\in W^{1,p}({M},r^{p+\beta})$. Then both $v_l\rightarrow u$ strongly  and $\nabla v_l\rightarrow X$ strongly in the corresponding $[\cdot]_{p,\mu}$-norms. Hence, $(v_l)$ is a Cauchy sequence in $W^{1,p}({M},r^{p+\beta})$, which implies $\nabla u=X$ and $u\in W^{1,p}({M},r^{p+\beta})$.
\end{proof}

Theorem \ref{weakconvergece} together with the dominated convergence theorem yields the following  two corollaries. We omit the proofs because they are standard but cumbersome.

\begin{corollary}\label{maxminlenew}
If $u\in {W^{1,p}}({{M}},{r}^{\beta+p})$, then $u_+:=\max\{u,0\}$, $u_-:=-\min\{u,0\}$ and $|u|=u_+-u_-$ are all in ${W^{1,p}}({{M}},{r}^{\beta+p})$.  In particular, $\nabla|u|=\sgn(u) \nabla u$.
\end{corollary}

\begin{corollary}\label{trunctionfun}
Given $u\in W^{1,p}(M,r^{p+\beta})$, we have $\max\{0,\min\{u,1\}\}\in W^{1,p}(M,r^{p+\beta})$. Moreover,
for any $\lambda>0$, set $u_\lambda:=\max\{-\lambda,\min\{u,\lambda\}\}$. Then $u_\lambda\rightarrow u$ in $ W^{1,p}(M,r^{p+\beta})$ as $\lambda \rightarrow \infty$.
\end{corollary}

\begin{proposition}\label{usefulproweakdre}Given $u,v\in W^{1,p}(M,r^{p+\beta})$, there following statements are true:
\begin{itemize}

\item[(i)] $\max\{u,v\}\in W^{1,p}(M,r^{p+\beta})$ and $|\nabla \max\{u,v\}|\leq \max\{|\nabla u |, |\nabla v | \}$;

\item[(ii)] if  $\|u\|_\infty+\|v\|_\infty<\infty$, then $uv\in W^{1,p}(M,r^{p+\beta})$ and $\nabla (uv)=u\nabla v+v\nabla u$.

\end{itemize}

\end{proposition}
\begin{proof}
  Since $\max\{a,b\}=\frac12(a+b+|a-b|)$, Statement (i) follows from Corollary \ref{maxminlenew} and Remark \ref{smoothdivege}/(a) immediately.
In order to prove Statement (ii), by Theorem \ref{MandM0isequal} one gets two sequences $(u_i),(v_i)$ in $C^\infty_0(M)$ such that $u_i\rightarrow u$ and $v_i\rightarrow v$ under $\|\cdot\|_{p,\beta}$. Set $u^*_i:= \max\{-\lambda,\min\{u_i,\lambda\}\}$ and $v^*_i:= \max\{-\lambda,\min\{v_i,\lambda\}\}$, where $\lambda$ is a constant satisfying $\|u\|_\infty+\|v\|_\infty\leq \lambda$. Thus,  $u^*_iv^*_i$ is a Lipschitz function with compact  support, which together with Lemma \ref{lipshccompact} implies $u^*_iv^*_i\in W^{1,p}(M,r^{p+\beta})$.
    Moreover,
    it follows from Corollary \ref{trunctionfun}  that $u^*_i,v^*_i\in W^{1,p}(M,r^{p+\beta})$ and  $u^*_i\rightarrow u$, $v^*_i\rightarrow v$ under $\|\cdot\|_{p,\beta}$. Since  $\|u^*_i\|_\infty+\|v^*_i\|_\infty\leq 2\lambda$,
  it is easy to check $u^*_iv^*_i\rightarrow uv$ and $\nabla(u^*_iv^*_i)\rightarrow v\nabla u+u\nabla v$ under the corresponding $[\cdot]_{p,\beta}$-norms.
  Hence, $(u^*_iv^*_i)$ is a Cauchy sequence in $W^{1,p}(M,r^{p+\beta})$. We conclude the proof by $u^*_iv^*_i\rightarrow uv$ in $\|\cdot\|_{p,\beta}$.
  \end{proof}

\begin{theorem}\label{pointswisestrongly}
Suppose that $(u_i)$ is a bounded sequence in $W^{1,p}({M},r^{p+\beta})$ and $u_i\rightarrow u$ pointwise a.e. in ${M}$. Thus, $u\in W^{1,p}({M},r^{p+\beta})$, $u_i\rightarrow u$ weakly in $L^p({M},r^{p+\beta})$ and $\nabla u_i\rightarrow \nabla u$ weakly in $L^p(T{M},r^{p+\beta})$.
\end{theorem}
\begin{proof}Choose an arbitrary subsequence $(u_{i_k})$ of $(u_i)$.
According to Theorem \ref{weakconvergece},  there is a subsequence  of $(u_{i_k})$, still denoted by $(u_{i_k})$, and a function $v\in W^{1,p}({M},r^{p+\beta})$ such that $u_{i_k}\rightarrow v$ weakly in $L^p({M},r^{p+\beta})$ and $\nabla u_{i_k}\rightarrow \nabla v$ weakly in $L^p(T{M},r^{p+\beta})$. We now show $v=u$ a.e. in ${M}$.

  Lemma \ref{Marzurlemma} together with the proof of Theorem \ref{weakconvergece} furnishes a sequence of convex combinations with
\[
u^*_{k}:=\sum_{j=k}^{m_k}a_{k,j}u_{i_j}\rightarrow v,\ \ \nabla u^*_k=\sum_{j=k}^{m_k}a_{k,j}\nabla u_{i_j}\rightarrow \nabla v, \text{ as }k\rightarrow \infty,
\]
in the corresponding $[\cdot]_{p,\mu}$-norms. Thus, the standard theory of $L^p$-space together with Lemma \ref{nullmeaure} yields a subsequence of $(u^*_k)$ converging  pointwise  to $v$ a.e. in ${M}$. On the other hand, since $(u_j)$ converges to $u$ pointwise, we see that $(u^*_k)$ converges to $u$ pointwise as well. Consequently,
  $v=u$ a.e. and therefore $\nabla u=\nabla v$ a.e. in ${M}$.

From above, we have proved that for every subsequence of $(u_i)$, there is a further subsequence $(u_{i_k})$ such that  $u_{i_k}\rightarrow u$ weakly in $L^p({M},r^{p+\beta})$ and $\nabla u_{i_k}\rightarrow \nabla u$ weakly in $L^p(T{M},r^{p+\beta})$, which indicates that the original sequence
$(u_i)$ also satisfies such a property.
\end{proof}

\begin{corollary}\label{maxepcro}
 Given $u\in W^{1,p}(M,r^{p+\beta})$, we have $\max\{u-\epsilon,0\}\in  W^{1,p}(M,r^{p+\beta})$ for any $\epsilon>0$.
\end{corollary}
\begin{proof}
Owing to Theorem \ref{MandM0isequal}, there exists a sequence $(u_i)$ in $C^\infty_0(M)$ such that $u_i\rightarrow u$ under $\|\cdot\|_{p,\beta}$ and $\|u_i\|_{p,\beta}\leq \|u\|_{p,\beta}+1$. By passing a subsequence, we may assume that $(u_i)$ converges to $u$ pointwise. Set $v_i:=\max\{u_i-\epsilon,0\}$. Since $v_i$ is a Lipschitz function with compact support, Lemma \ref{lipshccompact} yields $v_i\in W^{1,p}(M,r^{p+\beta})$. Moreover, note that $\|v_i\|_{p,\beta}\leq \|u_i\|_{p,\beta}\leq  \|u\|_{p,\beta}+1$ and
$(v_i)$ converges to $\max\{u-\epsilon,0\}$ pointwise, which together with Theorem \ref{pointswisestrongly} implies $\max\{u-\epsilon,0\}\in  W^{1,p}(M,r^{p+\beta})$.
\end{proof}

Now we introduce the capacity with respect to $\|\cdot\|_{p,\beta}$.

\begin{definition}The {\it Sobolev $(p,\beta)$-capacity} of a set $E\subset M$ is defined by
\begin{align*}
\Ca_{p,\beta}(E):=\inf_{u\in \mathscr{A}(E)}\|u\|^p_{p,\beta},
\end{align*}
where $\mathscr{A}(E)=\{ u\in {W^{1,p}}(M, {r}^{p+\beta}):\,u\geq 1 \text{ on  a neighborhood of }E \}$.
In particular, set $\Ca_{p,\beta}(E):=\infty$ if $\mathscr{A}(E)=\emptyset$.
\end{definition}

Proceeding as in the proof of Kinnunen and Martio \cite[Remark 3.1]{KM} and using Corollary \ref{trunctionfun}, we have the following result.

\begin{proposition}
Let $\mathscr{A}'(E):=\{u\in {W^{1,p}}(M, {r}^{p+\beta}):\,0\leq u\leq 1 \text{ and } u=1\text{ on  a neighborhood of }E \}$. Thus, we have
\[
\Ca_{p,\beta}(E)=\inf_{u\in \mathscr{A}'(E)}\|u\|^p_{p,\beta}.
\]
\end{proposition}

\begin{theorem}\label{propoertyfocap}
The  Sobolev $(p,\beta)$-capacity is an outer measure, that is
\begin{itemize}
\item[(i)] $\Ca_{p,\beta}(\emptyset)=0$;
\item[(ii)] $\Ca_{p,\beta}(E_1)\leq \Ca_{p,\beta}(E_2)$ if $E_1\subset E_2$;
\item[(iii)]
$\Ca_{p,\beta}\left( \bigcup_{i=1}^\infty E_i  \right)\leq \sum_{i=1}^\infty\Ca_{p,\beta}(E_i)$;
\item[(iv)] $\Ca_{p,\beta}(\cdot)$ is outer regular, i.e., $\Ca_{p,\beta}(E)=\inf\left\{ \Ca_{p,\beta}(O):\, O \text{ is open with }E\subset O \right\}$;
%\item[(v)] for an arbitrary Borel set $E\subset M$, we have $\mu(E)\leq \Ca_{p,\beta}(E)$.
\end{itemize}
\end{theorem}
\begin{proof}  Both (i) and (ii)  are obvious while (iv) follows from a similar argument to Kinnunen and Martio \cite[Remark 3.3]{KM}.
It remains to show (iii).
We may assume  $\sum_i\Ca_{p,\beta}(E_i)<\infty$. Given any $\epsilon>0$, for each $i\in \mathbb{N}$ there is  $u_i\in \mathscr{A}(E_i)$  such that
$\|u_i\|^p_{p,\beta}\leq \Ca_p(E_i)+{\epsilon}/{2^i}$.
Set $v:=\sup_iu_i$.

Now we show $v\in \mathscr{A}(\cup_i E_i)$.
In fact, let $v_k:=\max_{1\leq i\leq k}u_i$, $k\in \mathbb{N}$. Proposition \ref{usefulproweakdre}/(i) yields $v_k\in W^{1,p}(M,r^{p+\beta})$.
Moreover, since $|v_k|\leq \left|\sup_{i}u_i  \right|$ and $|\nabla v_k|\leq  \sup_i|\nabla u_i|$,
\begin{align*}
\|v_k\|^p_{p,\beta}
\leq  \int_M \sum_{i=1}^\infty|u_i|^p {\dmu}+\int_M \sum_{i=1}^\infty|\nabla u_i|^p {\dmu}=\sum_{i=1}^\infty\|u_i\|^p_{p,\beta}\leq \sum_{i=1}^\infty \Ca_{p,\beta}(E_i)+\epsilon<\infty.\tag{A.1}\label{estcapA.1}
\end{align*}
That is, $(v_k)$ is a bounded sequence in $W^{1,p}(M,r^{p+\beta})$.
Since $v_k\rightarrow v$ pointwise,  Theorem \ref{pointswisestrongly} implies $v\in W^{1,p}(M,r^{p+\beta})$. In particular,  it follows from Theorem \ref{pointswisestrongly}, Lemma \ref{Marzurlemma} and (\ref{estcapA.1}) that $\|v\|^p_{p,\beta}\leq \sum_{i=1}^\infty\|u_i\|^p_{p,\beta}\leq \sum_{i=1}^\infty \Ca_{p,\beta}(E_i)+\epsilon$.
On the other hand, because $u_i\in \mathscr{A}(E_i)$, there exists an open set $O_i\subset E_i$ such that $u_i\geq 1$ on $O_i$ and hence, $v=\sup_i u_i\geq 1$ in $\cup_{i=1}^\infty O_i$, which implies  $v\in \mathscr{A}(\cup_i E_i)$.  Therefore,
$\Ca_{p,\beta}\left( \cup_{i=1}^\infty E_i \right)\leq \|v\|^p_{p,\beta}\leq \sum_{i=1}^\infty \Ca_{p,\beta}(E_i)+\epsilon$,
which concludes the proof.
\end{proof}

%\begin{lemma}\label{zeromeasurebetween} For an arbitrary Borel set $E\subset M$, we have $\mu(E)\leq \Ca_{p,\beta}(E)$.
%\end{lemma}
%\begin{proof} The statement follows from a similar argument to Kinnunen et al. \cite[Lemma 4.1]{KM}.
%\end{proof}

\begin{definition}
We say that a property holds {\it $(p,\beta)$-quasieverywhere in a set $E$}, if it holds except for a subset of $(p,\beta)$-capacity zero in $E$.
In addition,
a function $u:M\rightarrow [-\infty,+\infty]$ is {\it $(p,\beta)$-quasicontinuous} if for any $\epsilon>0$, there is a set $E$ such that $\Ca_{p,\beta}(E)<\epsilon$ and the restriction  $u|_{M\backslash E}$ is continuous.
\end{definition}

\begin{theorem}\label{maintheoreminaapendix}
If $u\in W^{1,p}(M,r^{p+\beta})$ is $(p,\beta)$-quasicontinuous and $u=0$ $(p,\beta)$-quasieverywhere in $M\backslash {\Omega_N}$, then ${u|_{{\Omega_N}}}\in W^{1,p}_0({\Omega_N},r^{p+\beta})$.
\end{theorem}
\begin{proof}\textbf{Step 1.}
It suffices to show that $u$ can be approximated by the functions in $W^{1,p}_0({\Omega_N},r^{p+\beta})$.
Thanks to Corollary \ref{maxminlenew}, if we can find such a sequence for $u_+=\max\{u,0\}$, then we can also do it for $u_-=\max\{-u,0\}$. Thus, by $u=u_+-u_-$ we are done. For this reason, we may suppose $u\geq 0$.
Furthermore, since $W^{1,p}_0(M,r^{p+\beta})=W^{1,p}(M,r^{p+\beta})$ (see Theorem \ref{MandM0isequal}), we may assume that $u$ has a compact support in $M$. According to Corollary \ref{trunctionfun}, by considering truncation $u_\lambda=\min\{u,\lambda\}$, we may also assume that $u$ is bounded.

\smallskip

\noindent \textbf{Step 2.} Since $u$ is $(p,\beta)$-quasicontinuous, for any
$\delta>0$, Theorem \ref{propoertyfocap}/(iv) furnishes an open set $O\subset M$  such that $\Ca_{p,\beta}(O)<\delta$ and the restriction  $u|_{M\backslash O}$ is continuous. Denote
$E:=\{x\in M\backslash {\Omega_N}:\,u(x)\neq0  \}$.
The assumption implies $\Ca_{p,\beta}(E)=0$. Choose   a function $v_\delta\in \mathscr{A}'(O\cup E)$ satisfying $0\leq v_\delta\leq 1$ and $\|v_\delta\|^p_{p,\beta}\leq \Ca_{p,\beta}(O\cup E)+\eta$, where $\eta>0$ is a small constant with $\eta+\Ca_{p,\beta}(O)<\delta$. Thus,
\[
\|v_\delta\|^p_{p,\beta}\leq   \Ca_{p,\beta}(O)+\Ca_{p,\beta}(E)+\eta=\Ca_{p,\beta}(O)+\eta<\delta.\tag{A.2}\label{vbound}
\]
On the other hand, there is an open set $G\supset O\cup E$ such that $v_\delta=1$ in $G$. For any $\epsilon\in (0,1)$, define
\[
u_\epsilon(x):=\max\{u(x)-\epsilon,0 \}\in W^{1,p}(M,r^{p+\beta}).
\]
Given $x\in \partial {\Omega_N}\backslash G$, since $u(x)=0$, the continuity of  $u|_{M\backslash G}$ yields a $\rho_x>0$ such that  $u_\epsilon=0$ in $B_{\rho_x}(x)\backslash G$. Hence, $(1-v_\delta)u_\epsilon=0$ in $B_{\rho_x}(x)\cup G$ for every $x\in \partial{\Omega_N}\backslash G$. Now it follows from the compactness of  $\supp u$ that $(1-v_\delta)u_\epsilon$ is compactly supported in ${\Omega_N}$. Since $\|v_\delta\|_\infty+\|u_\epsilon\|_\infty<+\infty$, Proposition \ref{usefulproweakdre}/(ii) implies $-v_\delta u_\epsilon+u_\epsilon\in W^{1,p}(M,r^{p+\beta})$, which  together with Corollary \ref{compwp} yields  $((1-v_\delta)u_\epsilon)|_{{\Omega_N}} \in W^{1,p}_0({\Omega_N},r^{p+\beta})$. In the sequel, we show that $u$ can be approximated by this kind of functions  under $\|\cdot\|_{p,\beta}$.

\smallskip

\noindent \textbf{Step 3.}
By $\max\{a,b\}=\frac12(a+b+|b-a|)$ and Corollary \ref{maxminlenew},   one gets
\[
\nabla u_\epsilon=\nabla u \text{  a.e. in  }\{ x\in M:\, u(x)\geq \epsilon\};\quad \nabla u_\epsilon=0 \text{  a.e. in } \{ x\in M:\, u(x)\leq \epsilon\}.\tag{A.3}\label{weakgradientuesp}
\]
The triangle inequality furnishes
 \[
 \|u-(1-v_\delta)u_\epsilon\|_{p,\beta}\leq \|u-u_\epsilon\|_{p,\beta}+\|v_\delta u_\epsilon\|_{p,\beta}.\tag{A.4}\label{uv-1norem}
 \]
It follows from $u\geq 0$ that
$|u-u_\epsilon|=u-u_\epsilon\leq \epsilon$ and $\supp(u-u_\epsilon)\subset \supp u$. Hence,
 (\ref{weakgradientuesp}) and the dominated convergence theorem  yield
\begin{align*}
\|u-u_\epsilon\|_{p,\beta}\leq[u-u_\epsilon]_{p,\mu}+[\nabla u-\nabla u_\epsilon]_{p,\mu}\leq \epsilon [\chi_{\supp u}]_{p,\mu}+[\chi_{\{0\leq u\leq \epsilon\}}\nabla u]_{p,\mu}\rightarrow 0, \text{ as }\epsilon\rightarrow 0.\tag{A.5}\label{u-ue}
\end{align*}
On the other hand, since $u$ is bounded, Proposition \ref{usefulproweakdre}/(ii) and (\ref{vbound}) furnish
\begin{align*}
\|v_\delta u_\epsilon\|_{p,\beta}&\leq [v_\delta u_\epsilon]_{p,\mu}+[\nabla (v_\delta u_\epsilon)]_{p,\mu}\leq [v_\delta u_\epsilon]_{p,\mu}+[u_\epsilon \nabla v_\delta]_{p,\mu}+[v_\delta \nabla u_\epsilon]_{p,\mu}\\
&\leq \|u\|_\infty [v_\delta]_{p,\mu}+\|u\|_{\infty}[\nabla v_\delta]_{p,\mu}+[v_\delta \nabla u]_{p,\mu}\\
&\leq 2\|u\|_\infty \|v_\delta\|_{p,\beta}+[v_\delta \nabla u]_{p,\mu}\leq 2\delta^{\frac1p}\|u\|_\infty+[v_\delta \nabla u]_{p,\mu}.\tag{A.6}\label{vunorm}
\end{align*}
 Moreover, (\ref{vbound}) yields a subsequence $v_i:=v_{\delta_i}$ such that $v_i$ converges to $0$ pointwise a.e. (i.e., $\delta_i\rightarrow 0$) as $i\rightarrow \infty$. Owing to $\|v_\delta\|_\infty\leq 1$, the dominated convergence theorem then implies $\lim_{i\rightarrow \infty}[v_i \nabla u]_{p,\mu}=0$,
which together with (\ref{vunorm}) furnishes
\[
\lim_{i\rightarrow \infty}\|v_i u_\epsilon\|_{p,\beta}\leq\lim_{i\rightarrow \infty}\left( 2\delta_i^{\frac1p}\|u\|_\infty+[v_i \nabla u]_{p,\mu} \right)=0.\tag{A.7}\label{gradvu}
\]
Now it follows from (\ref{uv-1norem}), (\ref{u-ue}) and (\ref{gradvu}) that $ \|u-(1-v_i)u_\epsilon\|_{p,\beta}\rightarrow 0$ as $\epsilon\rightarrow 0$ and $i\rightarrow \infty$.
Since $((1-v_i)u_\epsilon)|_{{\Omega_N}}\in W^{1,p}_0({\Omega_N},r^{p+\beta})$, we have $u|_{{\Omega_N}}\in W^{1,p}_0({\Omega_N},r^{p+\beta})$.
\end{proof}

%Furthermore, by the method analogous to Kilpel\"ainen, Kinnunen, and Martio \cite[Theorem 4.6]{KKM}, one can
%show the following theorem.

%\begin{theorem}\label{spaceequalthe}
%Let $E$ be a closed subset in $M$. Thus, $W^{1,p}_0({\Omega_N},r^{p+\beta})=W^{1,p}_0({\Omega_N}\backslash E,r^{p+\beta})$ if and only if $\Ca_{p,\beta}(E)=0$. Here, $W^{1,p}_0({\Omega_N},r^{p+\beta})=W^{1,p}_0({\Omega_N}\backslash E,r^{p+\beta})$  means that every $u\in W^{1,p}_0({\Omega_N},r^{p+\beta})$ can be approximated by functions in $C^\infty_0({\Omega_N}\backslash E)$.
%\end{theorem}

\appendix
\noindent{\textbf{Funding.}}
This work was supported by National Natural Science Foundation of China (No. 11761058) and Natural Science Foundation of Shanghai (No. 19ZR1411700).

\smallskip

\appendix
\noindent{\textbf{Acknowledgements.}}
The third author is greatly indebted to Prof. A. Krist\'aly for many useful discussions and helpful comments.

\end{document}